\documentclass{amsart}

\title{Local and global integrability of Lie brackets}

\author[Rui L Fernandes]{Rui Loja Fernandes}
\email{ruiloja@illinois.edu}
\author{Yuxuan Zhang}
\email{yzhng172@illinois.edu}
\address{Department of Mathematics, University of Illinois at Urbana-Champaign, 1409 W. Green Street, Urbana, IL 61801, USA}
\thanks{This work was partially supported by NSF grants DMS-1710884 and DMS-2003223.}
\date{\today}

\usepackage[T1]{fontenc}
\usepackage{lmodern}
\usepackage{microtype}
\usepackage{enumerate}
\usepackage{hyperref}
\usepackage{bm}
\usepackage{tikz}
\usetikzlibrary{patterns,decorations.pathreplacing}
\usetikzlibrary{decorations.markings}
\usetikzlibrary{matrix,arrows,calc}
\usetikzlibrary{arrows.meta}
\usepackage[utf8]{inputenc}
\usepackage{graphicx}
\usepackage{amsfonts, amsmath, amsthm, amssymb}
\usepackage[percent]{overpic}
\usepackage{tensor}
\usepackage{comment}
\usepackage[all]{xy}
\usepackage{listings}
\usepackage{placeins}
\usepackage{mathrsfs}

\newcommand{\R}{\ensuremath{\mathbb{R}}}
\newcommand{\C}{\ensuremath{\mathbb{C}}}

\renewcommand{\S}{\ensuremath{\mathbb{S}}}

\newcommand{\G}{\mathcal{G}}
\newcommand{\s}{\mathbf{s}}             
\renewcommand{\t}{\mathbf{t}}           
\newcommand{\gm}{\mathbf{m}}           
\newcommand{\gu}{\mathbf{u}}		
\newcommand{\ginv}{\bm{\iota}} 		

\renewcommand{\gg}{\mathfrak{g}}
\newcommand{\g}{\mathfrak{g}}
\renewcommand{\H}{\mathcal{H}}
\newcommand{\NN}{\mathcal{N}}
\newcommand{\V}{\mathcal{V}}
\newcommand{\U}{\mathcal{U}}
\newcommand{\X}{\ensuremath{\mathfrak{X}}}
\newcommand{\AsCo}{\mathcal{AC}}

\newcommand{\timesst}{\tensor[_\s]{\times}{_\t}}

\newcommand{\tto}{\rightrightarrows}

\newcommand{\Lie}{\mathscr{L}}          

\DeclareMathOperator{\id}{Id}

\DeclareMathOperator{\Assoc}{Assoc}

\DeclareMathOperator{\Ker}{Ker}
\renewcommand{\d}{\mathrm d}               
\DeclareMathOperator{\pr}{pr}      
\DeclareMathOperator{\ad}{ad}           
\DeclareMathOperator{\LieA}{Lie}     

\newcommand{\al}{\alpha}
\newcommand{\be}{\beta}

\newcommand{\pd}[1]{\frac{\partial}{\partial #1} }   
\newcommand{\Lvf}[1]{\overleftarrow{#1}} 

\newtheorem{theorem}{Theorem}[section]
\newtheorem{lemma}[theorem]{Lemma}
\newtheorem{corollary}[theorem]{Corollary}
\newtheorem{proposition}[theorem]{Proposition}
\newtheorem{example}[theorem]{Example}

\theoremstyle{definition}
\newtheorem{definition}[theorem]{Definition}
\newtheorem{remark}[theorem]{Remark}

\begin{document}

\begin{abstract}
We survey recent results on the local and global integrability of a Lie algebroid, as well as the integrability of infinitesimal multiplicative geometric structures on it.  
\end{abstract}

\dedicatory{In memory of Kirill Mackenzie}

\maketitle

\setcounter{tocdepth}{1}
\tableofcontents

\section{Introduction}
Our aim in this paper is to present an approach to integrability problems in the context of the Lie theory for algebroids and groupoids consisting of two steps:
\smallskip
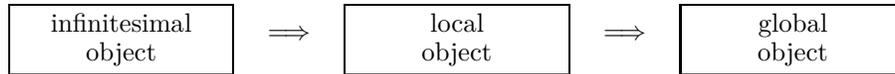
\[ \framebox[3 cm]{\txt{infinitesimal\\ object}}\quad \Longrightarrow \quad   \framebox[3 cm]{\txt{local\\ object}}  \quad \Longrightarrow \quad   \framebox[3 cm]{\txt{global\\ object}} \]
\smallskip
Two examples of this are:
\begin{enumerate}[1)]
\item Given a Lie algebroid, one first integrates it to a local Lie groupoid -- which is always possible. One then tries to enlarge the local Lie groupoid to a global Lie groupoid.
\item Given a Poisson manifold, one first integrates it to a local symplectic groupoid -- which is always possible. Assuming that the local Lie groupoid extends to a global Lie groupoid, one then tries to extend the symplectic form to the global Lie groupoid.
\end{enumerate}

The idea of using a local-global approach is very old in classical Lie theory. One illustration is the use of the Baker-Campbell-Hausdorff formula to prove the Lie correspondence. However, in the case of algebroids and groupoids, due to the failure of Lie's Third Theorem, the local-global approach was understood only recently, thanks to the contributions of several authors \cite{BG16,BC12,CMS20,CMS20b,FM20,Yudilevich16}.  

This local-global approach is an alternative to the $A$-path construction approach first used in \cite{CrainicFernandes:integrability} to study the failure of Lie's Third Theorem for Lie algebroids. Although the $A$-path construction has proved to be extremely successful, we believe that the local-global approach provides a complementary perspective, gives new insights into integrability problems, and it deserves to be widely known.

The subject of this paper is very much related to Kirill Mackenzie's mathematical life. In the next section we will provide a personal historical account about these connections, as well as a detailed description of the contents and results contained in this paper.
\smallskip

\textbf{Acknowledgments.} The ideas and results presented in this paper were influenced by many conversations we had along several years with Henrique Bursztyn, Ale Cabrera, Marius Crainic, Ioan Marcut and Daan Michiels. We are indebted to all of them. We also thank the two referees for point out several errors and typos in a first version of this paper.

\section{Integrability problems and Mackenzie's contributions}
The main problem that initially drove the mathematical life of Kirill Mackenzie was the validity of Lie's Third Theorem for Lie algebroids. In the late 1960's, Pradines had sketched in a series of short papers published in the Comptes Rendus de l'Acad\'emie des Sciences de Paris \cite{Pradines66,Pradines67,Pradines68} a Lie theory for Lie algebroids and Lie groupoids. In particular, Pradines claimed that every Lie algebroid integrates to a global groupoid. With the aim of giving a complete proof of this statement, Mackenzie developed a strategy in the style of the Cartan and van Est \cite{VanEst88} cohomological proof for Lie algebras and Lie groups. He found a cohomological obstruction to integrability of \emph{transitive} Lie algebroids and he tried to show for some time that this obstruction vanished. During this period he learned that Almeida and Molino \cite{AM85}, while studying transversally parallelizable foliations, had found that the statement is actually false: there are Lie algebroids which do not integrate to Lie groupoids. Still, Mackenzie's strategy paid some dividends, since his obstruction allowed to decide which transitive Lie algebroids were actually integrable, as he explained in detail in his first monograph \cite{Mackenzie}.

The obstructions to integrability for a general (i.e., not necessarily transitive) Lie algebroid took another 20 years to be fully understood, and was finally settled by Crainic and Fernandes in \cite{CrainicFernandes:integrability}. The solution proposed in  \cite{CrainicFernandes:integrability} rests on the so-called $A$-path space construction. Given any Lie algebroid $A$ one constructs a $A$-path homotopy groupoid:
\[ \G(A)=\frac{\{A\textrm{-paths}\}}{A\text{-path homotopy}}, \]
This is a topological groupoid which has the following properties:
\begin{enumerate}[(i)]
\item $A$ is integrable if and only if $\G(A)$ inherits a quotient smooth structure from the space of $A$-paths, and
\item the obstructions to smoothness can be expressed in terms of certain monodromy groups. 
\end{enumerate}
In the transitive case, one can recover from this approach Mackenzie's cohomological obstruction. The $A$-path space construction has its roots, on the one hand, on the Poisson sigma model approach to the integrability of Poisson manifolds due to Cattaneo and Felder \cite{CaFe} and, on the other hand, on ideas of Severa \cite{Severa05} inspired by Sullivan's rational homotopy theory. Another source of inspiration for \cite{CrainicFernandes:integrability} was a similar path approach due to Duistermaat and Kolk for the special case of Lie algebras that appeared in their monograph \cite{DuistermaatKolk00}.

Several other integrability problems have filled Mackenzie's mathematical life. For example, a Lie algebroid often carries an additional geometric structure and one is interested in integrating it along with the Lie algebroid. One paradigmatic example is the cotangent Lie algebroid $T^*M$ of a Poisson manifold $(M,\pi)$. Weinstein \cite{WeinstGpd} and Karasev \cite{KM93} had found that the global object corresponding to a Poisson manifold $(M,\pi)$ is a symplectic groupoid $(\G,\Omega)$, meaning a Lie groupoid $\G\tto M$ together with a multiplicative symplectic form $\Omega\in\Omega^2(\G)$. The Lie groupoid $\G\tto M$ is an integration of $T^*M$, so the integrability of $T^*M$ is an obstruction to integrate a Poisson manifold $(M,\pi)$. The remaining question was then if a groupoid integrating $T^*M$ carries a multiplicative symplectic form. An answer to this question was provided by Mackenzie and Xu:

\begin{theorem}[MacKenzie and Xu \cite{MX00}]
A target 1-connected Lie groupoid integrating the cotangent bundle of a Poisson manifold $(M,\pi)$ is automatically a symplectic groupoid.
\end{theorem}

Note that the 1-connectedness assumption cannot be avoided: there are examples of Lie groupoids integrating the cotangent bundle of a Poisson manifold $(M,\pi)$ which do not carry any symplectic form. Mackenzie's second monograph \cite{Mackenzie:general} contains a detailed discussion of this result. One can also use the $A$-path space construction to describe explicitly the multiplicative symplectic form $\Omega$ on $\G(T^*M)$ -- see \cite{CrainicFernandes:integrability:Poisson} -- and Cattaneo and Felder in \cite{CaFe} showed that it can be viewed as an infinite dimensional symplectic quotient.

More generally, one can ask if an \emph{infinitesimal} multiplicative geometric structure, such as a differential form, a multivector field, a tensor field, a connection, etc, on an integrable Lie algebroid $A$ can be integrated to a multiplicative geometric structure on some groupoid $\G$ integrating $A$. In the last 15 years many such results have been proved (e.g., \cite{BC12,BCO09,BCWZ04,BD19,CSS15,CX04,DE19,ILX12,JO14,LSX09,Ortiz13}). Here, we will focus mainly on the case of forms for which we have the following generalization of the result of Mackenze and Xu:

\begin{theorem}[\cite{BC12,BCWZ04}]
\label{thm:integration:forms:1:connected}
Let $\G\tto M$ be a target 1-connected Lie groupoid integrating a Lie algebroid $A\to M$. There is a 1:1 correspondence:
\[ 
\left\{\txt{infinitesimal multiplicative\\ $k$-forms on $A$ \,}\right\}
\tilde{\longleftrightarrow}
\left\{\txt{multiplicative\\ $k$-forms on $\G$ \,} \right\}
\]
\end{theorem}

Not surprisingly, this type of result can also be proved using the $A$-path construction and, ultimately, it amounts to the integration of an algebroid morphism to a groupoid morphism (see \cite{BC12}). 
\smallskip

As stated in the Introduction, our aim here is to present a local-global approach to such integrability problems. This approach starts with the following key result:

\begin{theorem}[\cite{CrainicFernandes:integrability}]
\label{thm:local:integration:grpd}
Every Lie algebroid integrates to a \emph{local} Lie groupoid. 
\end{theorem}

Here by ``local" one means that the multiplication is only defined for arrows close enough to the the identity arrows (see Section \ref{sec:local:groupoid} for the precise definition). This result was also first announced by Pradines \cite{Pradines68}, but he never published a proof. If, e.g., one tries to extend the known result from Lie groups one faces the problem of finding a Baker-Campbell-Hausdorff formula. In the special context of Poisson manifolds, the existence of a local symplectic groupoid had been proved before in \cite{CDW87}, using local symplectic realizations. The original proof of Theorem \ref{thm:local:integration:grpd} in \cite{CrainicFernandes:integrability} uses an $A$-path space construction. More recently, in \cite{CMS20b} a new proof for Lie algebroids was found which can be viewed as a generalization of the Coste, Dazord and Weinstein proof: using a Lie algebroid connection $\nabla$ on $A$ the authors construct a local Lie groupoid in a neighborhood of the zero section of $A$, which they call the \emph{spray groupoid}, and which we will recall in Section \ref{sec:local:integration:grpd}.

The next obvious step is to solve the question of whether a local Lie groupoid is contained in a (global) Lie groupoid, i.e., whether a local Lie groupoid is ``globalizable'' -- sometimes called ``enlargeable''. In the case of Lie groups this problem was first studied by Mal'cev \cite{Malcev41} (see also \cite{Olver:Lie}) who showed that the lack of associativity is the only obstruction to embedding a local group in a global one. Recently, Mal'cev's Theorem was extended to realm of local Lie groupoids:

\begin{theorem}[Mal'cev's theorem for groupoids \cite{FM20}]
    A local Lie groupoid is globalizable if and only if its multiplication is globally associative.
\end{theorem}

Let us explain what is the global associativity property in the statement. The associativity axiom for a local groupoid requires that for every triple of composable arrows one has:
\[ g(hk)=(gh)k, \]
provided both sides are defined. While for a (global) Lie groupoid this implies that associativities involving 4 or more elements also hold, this is not true for a local Lie groupoid. So, for example, there exist local Lie groupoids (\cite{FM20,Olver:Lie}) in which one can find 4 elements such
that:
\[ (gh)(kl)\not = g(hk)l, \]
so that 4-associativity does not hold -- see the discussion in Section \ref{sec:associativity}. An obvious necessary condition for a local groupoid to be globalizable is that $n$-associativity holds for all $n\ge 3$, in which case we say that the local Lie groupoid is \emph{globally associative}. Mal'cev's theorem states that this condition is also sufficient.

The proof of Mal'cev's theorem uses the notion of  {\em associative completion} of a local groupoid. Given a local groupoid $G\tto M$ its associative completion is a groupoid $\AsCo(G)\tto M$ obtained by considering the well-formed words with alphabet in $G$ modulo the equivalent relation generated by contractions:
\[ (g_1,\dots,g_i,g_{i+1},\dots,g_n)\sim (g_1,\dots,g_ig_{i+1},\dots,g_n). \]
Actually, $\AsCo(-)$ is a functor: any morphism $\phi:G\to H$ of local groupoids yields a groupoid morphisms $\AsCo(\phi):\AsCo(G)\to\AsCo(H)$ -- see Section \ref{sec:associative:completion} for more details.  

In general, if $G$ is a local Lie groupoid, $\AsCo(G)$ is not a Lie groupoid, it is only a topological groupoid.  An element in the isotropy 
\[ g\in G_x=\s^{-1}(x)\cap \t^{-1}(x) \] 
is called an \emph{associator} at $x$ if there is a well-formed word $(g_1,\dots,g_n)$ which admits two sequences of contractions: one ending at $g$ and the other one ending at the identity $1_x$. The set of all associators, denoted $\Assoc(G)$, is contained in the kernel of the completion map $G\to \AsCo(G)$, and one of the main results in \cite{FM20} shows that the associators control the smoothness of $\AsCo(G)$:

\begin{theorem}[\cite{FM20}]
    For a local Lie groupoid $G$ the completion $\AsCo(G)$ is a Lie groupoid if and only if $\Assoc(G)$ is uniformly discrete in $G$.
\end{theorem}

The functor $\AsCo(-)$ shares obvious similarities with the integration functor $\G(-)$, which associates to a Lie algebroid the space of $A$-paths modulo $A$-homotopies.  This is pursued much further in \cite{FM20}, which contains a result giving the precise relationship between the associators and the monodromy groups of \cite{CrainicFernandes:integrability}, expressing the obstructions to integrability. This will be explained briefly in Section \ref{sec:associators}. 

All together, the previous results establish the local-global path to integrability of Lie algebroids. So let us turn now to the integrability of infinitesimal multiplicative structures. One of the advantages of the local integration of a Lie algebroid $A$ by a spray groupoid $G$, is that it allows to find an explicit integration of any infinitesimal multiplicative form, and any other infinitesimal multiplicative objects. So one has the first step in the local-global path to integrability of such objects:

\begin{theorem}[\cite{CMS20}]
If $G\tto M$ is a spray groupoid of a Lie algebroid $A\to M$ associated with an $A$-connection, there is an \textbf{explicit} 1:1 correspondence:
\[ 
\left\{\txt{infinitesimal multiplicative\\ $k$-forms on $A$ \,}\right\}
\tilde{\longleftrightarrow}
\left\{\txt{multiplicative\\ $k$-forms on $G$ \,} \right\}
\]
\end{theorem}

Here by \emph{explicit} we mean that there are local formulas describing the multiplicative form in terms of the IM-form -- see Theorem \ref{thm:local:formula}.

Finally, for the second step in the local-global path to integrability infinitesimal multiplicative objects, one needs to find if one can extend a multiplicative structure on a local groupoid $G$ to a multiplicative structure on its completion $\AsCo(G)$, assuming that the completion is smooth.  Since $\AsCo(G)$, in general, will not have 1-connected target fibers one cannot apply Theorem \ref{thm:integration:forms:1:connected}. However, perhaps somewhat surprisingly, one finds that one can always extend such objects:

\begin{theorem}[\cite{BG16}]
    Let $G$ be a local Lie groupoid and assume that $\Assoc(G)$ is uniformly discrete in $G$. Then there is a 1:1 correspondence:
\[ 
\left\{\txt{multiplicative\\ $k$-forms on $G$ \,}\right\}
\tilde{\longleftrightarrow}
\left\{\txt{multiplicative\\ $k$-forms on $\AsCo(G)$ \,} \right\}
\]
\end{theorem}

We will discuss this result, including a detailed proof, in the last section of the paper. There we will also discuss how this result is related to the integrability of IM-forms for Lie groupoids with 1-connected $\t$-fibers stated in Theorem \ref{thm:integration:forms:1:connected}, and see that it leads to an independent proof of this result.
\medskip

\medskip

{\bf Notations and Conventions.}
Throughout the text, local groupoids will be denoted by Latin letters (e.g.\ $G$),
and global groupoids will be denoted by calligraphic versions (e.g.\ $\G$).
If $A$ is a Lie algebroid, then $\G(A)$ will denote the associated target-simply connected
groupoid (it is global, so it has a calligraphic symbol.)
Arrows of a groupoid compose from right to left, so that multiplication 
$gh=\gm(g,h)$ is defined if $\s(g) =\t(h)$. We also denote the identity section by $\gu:M\to\G$ and the inverse map by $\ginv:\G\to\G$. The Lie algebroid of a Lie groupoid is defined using left-invariant vector fields, so $A=\ker\d_M \t$. Finally, for a symplectic groupoid $(\G,\Omega)\tto M$ the Poisson structure on the base is the unique one making the target a Poisson map. We use as a general reference for Lie groupoids and Lie algebroids the Lecture Notes \cite{CrainicFernandes:lectures}, but note that they follow conventions different from ours.

\section{From Lie groupoids to Lie algebroids}
\label{sec:preliminaries:groupoids}

Given a Lie groupoid $\G\tto M$ with Lie algebroid $(A,[\cdot,\cdot]_A,\rho)$ an $A$-connection $\nabla$ defines an exponential map $\exp_\nabla:A\to \G$ which is a diffeomorphism in an open set $V\subset A$ containing the zero section. In this section we explain how to use this exponential map to obtain (i) a (local) groupoid structure on $V$ from the groupoid structure on $\G$ and (ii) multiplicative geometric structures on $V$ from multiplicative geometric structures on $\G$.

\subsection{Connections, sprays and exponential map}
Let  $(p:A\to M,[\cdot,\cdot]_A,\rho)$ be a Lie algebroid. Recall that an \emph{$A$-path} is a path $a:I\to A$, defined on some interval $I$, satisfying:
\[ \rho(a(t))=\frac{\d}{\d t}p(a(t)),\quad \forall t\in I. \]
We denote by $\nabla$ an $A$-connection on $A$ -- see \cite{Fernandes:holonomy} -- so that $\nabla:\Gamma(A)\times\Gamma(A)\to \Gamma(A)$ is $\R$-bilinear and satisfies:
\[ \nabla_{f\al}\be=f\nabla_{\al}\be,\qquad \nabla_{\al}f\be=f\nabla_{\al}\be+\rho(\al)(f)\be, \]
for all $f\in C^\infty(M)$ and $\al,\be\in\Gamma(A)$. Lie algebroids always carry such connections and one can choose $\nabla$ to be torsion-free, i.e., such that the torsion
\[ T^\nabla(\al,\be):=\nabla_\al\be-\nabla_\be\al-[\al,\be]_A, \]
vanishes identically.

Given an $A$-connection $\nabla$ the geodesics are the $A$-paths $a:I\to A$ satisfying:
\[ \nabla_a a=0. \]
They are the integral curves of a vector field $X_\nabla\in\X(A)$ called the geodesic spray of the connection. Fixing a chart $(U,x^i)$ for $M$ and local basis of sections $\{e_r\}$ for $\Gamma_U(A)$ one has local coordinates $(x^i,\xi^r)$ for the total space $A$. Then one finds the following local expression for the spray:
\begin{equation}
\label{eq:geodesic:spray}
X_\nabla=\sum_{i,k} \rho^{i}_k(x) \xi^k \pd{x^i}-\sum_{k,l,m}\Gamma_{kl}^m(x)\, \xi^k\, \xi^l \pd{\xi^m},
\end{equation}
where the coefficients are defined from the anchor and the connection by:
\[ \rho(e_k)=\sum_{i} \rho^{i}_k(x) \pd{x^i},\qquad \nabla_{e_k}e_l=\sum_{m}\Gamma_{kl}^m(x) e_m. \]
Using this expression, one sees that the spray $X_\nabla\in\X(A)$ satisfies two basic properties:
\begin{enumerate}
\item[(S1)] $\d_{a} p (X_\nabla|_a)=\rho(a)$, for all $a\in A$,
\item[(S2)] $(m_t)_*X_\nabla=\frac{1}{t}X_\nabla$, for all $t>0$,
\end{enumerate}
where $p:A\to M$ denotes the projection and $m_t:A\to A$ is the scalar multiplication by $t\in\R$. These properties completely characterize the spray, and in fact one checks easily that:

\begin{lemma}
Given a Lie algebroid $A\to M$, there is a 1:1 correspondence between torsion free $A$-connections and vector fields $X\in\X(A)$ satisfying (S1) and (S2).
\end{lemma}

We will also need the global version of (S1) and (S2) satisfied by the geodesic flow $\phi^t_{X_\nabla}$, i.e., the flow of $X_\nabla$. They are:
\begin{enumerate}
\item[(GS1)] Integral curves $t\mapsto \phi_{X_\nabla}^t(a)$ are $A$-paths;
\item[(GS2)] Whenever defined:
\begin{equation}\label{eq:flow_spray}
\phi_{X_\nabla}^t\circ m_s= m_s\circ\phi_{X_\nabla}^{st}\quad\quad (t,s\in\R).
\end{equation}
\end{enumerate}

If $\G\tto M$ is a Lie groupoid with Lie algebroid $A\to M$, an $A$-connection $\nabla$ determines a partial connection $\widetilde{\nabla}$ along the $\t$-fibers, i.e, a family of connections on the $\t$-fibers. It is completely determined by requiring that on left-invariant vector fields one has:
\[ \widetilde{\nabla}_{\Lvf{\al}}\Lvf{\be}=\Lvf{\nabla_\al\be}, \quad (\al,\be\in\Gamma(A)). \]
Then one can define the groupoid exponential map from the (ordinary) exponential maps of these fibers connections:
\[ \exp_\nabla: A\to \G,\quad a\mapsto \exp_{\widetilde{\nabla}}(a). \] 
Of course, in general, $\exp_\nabla$ is only defined in some neighborhood of the zero section of $A$. 

Notice that the partial connection $\widetilde{\nabla}$ restricts to a left-invariant connection on each isotropy group $\G_x$. The initial connection $\nabla$ can be chosen so that this restriction satisfies $\widetilde{\nabla}_X X=0$, for every left-invariant vector field $X\in\X(\G_X)$. In this case the exponential map restricts on the isotropy $\gg_x\subset A$ to the ordinary Lie group exponential map:
 \[ 
\xymatrix{
\G_x\ \ar@{^{(}->}[r] & \G \\
\gg_x\ \ar[u]^{\exp} \ar@{^{(}->}[r] & A\ar[u]_{\exp_\nabla}}
\]

\subsection{Maurer-Cartan form on a Lie groupoid}
Let us recall from \cite{FernandesStruchiner1} the construction of the (left) Maurer-Cartan form on the Lie groupoid $\G\tto M$. It is the $\t$-foliated 1-form with values in $A$ given by:
\[ \theta^\G\in\Omega^1(T^\t \G,\s^*A),\quad \theta^\G(V):=\d_g L_{g^{-1}} V \quad (V\in T^\t_g \G:=\ker\d_g\t). \]
Note that one can also viewed $\theta^\G$ as a bundle map covering the source map:
\begin{equation}
\label{eq:MC:bundle:map}
\vcenter{
\xymatrix{
T^\t\G\ar[d]_{\pr} \ar[r]^---{\theta^\G} & A\ar[d]^p \\
\G\ar[r]_{\s} & M}}
\end{equation}
which is a fiberwise linear isomorphism. 

The Maurer-Cartan form allows one to identify the left-invariant vector fields among all vector fields in $\G$ tangent to the $\t$-fibers. In fact, the following lemma is obvious from its definition:

\begin{lemma}
\label{lem:MC:form}
The left-invariant vector field determined by a section $\al\in\Gamma(A)$ is the unique vector field $\overleftarrow{\al}\in\X(\G)$ satisfying:
\[ i_{\overleftarrow{\al}}\theta^\G=\s^*\al,\quad \d\t(\overleftarrow{\al})=0. \]
\end{lemma}

Similarly, the Maurer-Cartan form also allows to transport $A$-paths to groupoid paths: given an $A$-path $a:I\to M$ with initial point $x\in M$, there is a unique path $g:I\to \G$ in the fiber $\t^{-1}(x)$ such that:
\[ \theta(\dot{g}(t))=a(t), \quad g(0)=1_x. \]
In particular, given an $A$-connection $\nabla$ its exponential map is given by:
\[ \exp_\nabla: A\to \G,\quad \exp_\nabla(a)=g(1), \]
where $g(t)$ is the path associated with $A$-path $t\mapsto \phi_{X_\nabla}^t(a)$.  Hence, if for $t\in\R$, one defines $\exp^{t}_\nabla:A\to \G$ by setting:
\[ \exp_\nabla^t(a):=\exp_\nabla(ta), \]
one concludes that we have a commutative diagram:
\begin{equation}
\label{eq:diagram:flows:exponential}
\vcenter{
\xymatrix{
\G & \ker\d\t \ar[l]_{\pr} \ar[r]^{\theta^\G}  & A \\
& A \ar[u]^--{\phi^t_{\tilde{X}_{\nabla}}} \ar[ul]^--{\exp_\nabla^t} \ar[ur]_--{\phi^t_{X_\nabla}}}}
\end{equation}
where $\widetilde{X}_{\nabla}$ is the spray of the partial connection $\widetilde{\nabla}$ on the $\t$-fibers.

In the case of a Lie group $\G$ with Lie algebra $\gg$ there is a well-known formula for the pullback of the Maurer-Cartan form $\theta^\G$ under the exponential map:
\[ \big(\exp^*\theta^\G\big)_a(b)=\int_0^1 e^{t\ad_a} b\, \d t. \]
This formula has been extended to the case of Lie algebroids in \cite{Yudilevich16}. For that we need to recall that the flow of time-dependent section of a Lie algebroid $\al_t\in\Gamma(A)$ is a family of Lie algebroid automorphisms $\{\phi^{t,s}_{\al_t}\}$ covering the flow of the time-dependent vector field $\rho(\al_t)$:
\[
\xymatrix@R=15pt{
A\ar[d] \ar[r]^---{\phi^{t,s}_{\al_t}} & A\ar[d] \\
M\ar[r]_{\phi^{t,s}_{\rho(\al_t)}} & M}
\]
characterized by:
\[ \frac{\d}{\d t} (\phi^{t,s}_{\al_t})^*(\be)=(\phi^{t,s}_{\al_t})^*([\al,\be]),\quad \phi^{t,t}_{\al_t}=\id, \quad \quad (\be\in\Gamma(A)). \]

\begin{theorem}[\cite{Yudilevich16}]
\label{thm:MC:exponential}
Let $\G\tto M$ be a Lie groupoid with Lie algebroid $p:A\to M$ and let $\nabla$ be an $A$-connection. Then for any $a$ in a neighborhood $V\subset A$ of the zero section and any $b\in A_{p(a)}$ one has:
\[ \big(\exp_\nabla^*\theta^\G\big)_a(b)=\int_0^1 \phi^{1,t}_{\al_{t,0}}\frac{\d}{\d s}\Big|_{s=0}\al_{t,s}(p(\phi^t_{X_\nabla}(a)))\, \d t, \]
where $\al_{t,s}\in\Gamma(A)$ is any family of sections such that:
\[ \al_{t,s}(p(\phi^t_{X_\nabla}(a+sb)))=\phi^t_{X_\nabla}(a+sb). \]
\end{theorem}


\subsection{Groupoid tubular neighborhoods}
\label{subsection:groupoid:exponential}
Let $\G$ be a Lie groupoid with Lie algebroid $A$ and fix an $A$-connection $\nabla$. One can choose a neighborhood $V\subset A$ of the zero section where the exponential map restricts to a diffeomorphism onto an open $W\subset \G$ containing the units:
\[ A\supset \xymatrix{V\ar[rr]^{\exp_\nabla}_{\cong}&& W}\subset \G \]
Using this diffeomorphism, one can induce a (local) groupoid structure on $V$. First, by further restricting $V$, one can assume that $W$ is preserved by the groupoid inversion. Then we obtain source, target, inverse and unit maps on $V$ such that
\[
\xymatrix{V \ar@(dl, ul)^{\ginv} \ar@<0.25pc>[d]^{\s} \ar@<-0.25pc>[d]_{\t} \ar[r]^--{\exp_\nabla} &\G \ar@(dr, ur)_{\ginv} \ar@<0.25pc>[d]^{\s} \ar@<-0.25pc>[d]_{\t} \\ M\ar@/_1pc/[u]_{\gu} \ar@{=}[r] & M \ar@/_1pc/[u]_{u}}
\]
They are given purely in terms of Lie algebroid data as follows: the open set $V\subset A$ is invariant under the map $a\mapsto -\phi^1_{X_\nabla}(a)$ and one has:
\begin{itemize}
\item[-] The unit section is the zero section $\gu:M\to V\subset A$, $x\mapsto 0_x$;
\item[-] The target map is the bundle projection $\t:=p:V\to M$;
\item[-] The source map is $\s:=p\circ\phi^1_{X_\nabla}:V\to M$;
\item[-] The inverse map $\ginv: V\to V$ is given by $a\mapsto -\phi^1_{X_\nabla}(a)$.
\end{itemize}

Finally, one needs to push the multiplication to $V$. It will be defined only on the open set:
\[ \U:=\{(v_1,v_2)\in V\timesst V: \gm(\exp_\nabla(v_1),\exp_\nabla(v_2))\in\exp_\nabla(V)\}. \]
and it is a map $\gm:\U\to V$ making the following diagram commute:
\[
\xymatrix{\U \ar[d]_{\gm} \ar[rr]^--{\exp_\nabla\times\exp_\nabla} & & \G\timesst\G \ar[d]^{\gm} \\ 
V\ar[rr]_{\exp_\nabla} & & \G}
\]
Hence, we obtain a \emph{local} Lie groupoid. These will be discussed in the next section. 

For now, just like for the other structure maps, we would like to express multiplication in $V$ purely in terms of Lie algebroid data. For that we can apply Theorem \ref{thm:MC:exponential} to first express the pullback of the Maurer-Cartan form:
\[ \theta^A:=\exp_{\nabla}^*\theta^\G, \]
exclusively in terms of Lie algebroid data. This allows us to define left-invariant vector fields in $V$, without referring to $\G$: given $\al\in\Gamma(A)$ the corresponding left-invariant vector field $\overleftarrow{\al}\in\X(A)$ is the unique vector field satisfying:
\[ 
\left\{
\begin{array}{l}
i_{\overleftarrow{\al}}\theta^A=\s^*\al \\
\\
\d\t(\overleftarrow{\al})=0 
\end{array}
\right.\quad \Longleftrightarrow\quad 
\left\{
\begin{array}{l}
i_{\overleftarrow{\al}}\theta^A=(p\circ\phi^1_{X_\nabla})^*\al \\
\\
\d p(\overleftarrow{\al})=0 
\end{array}
\right.
\]
To define the product of  $v_1,v_2\in\U$ with $\s(v_1)=p\circ\phi^1_{X_\nabla}(v_1)=p(v_2)=\t(v_2)$, we first find a section $\al\in\Gamma(A)$ such that:
\[ \phi^1_{\overleftarrow{\al}}(0_{p(v_2)})=v_2, \]
and then set:
\begin{equation}
\label{eq:multiplication:exponential}
\gm(v_1,v_2):=\phi^1_{\overleftarrow{\al}}(v_1).
\end{equation}
Notice that this uses only algebroid data.

\subsection{Local formulas for multiplicative stuctures}
Given some geometric structure on a Lie groupoid $\G\tto M$ one can use the exponential map of a connection to transfer it to a neighborhood of the identity section of the Lie algebroid $A\to M$ of $\G$. In general, the resulting structure cannot be described purely in terms of Lie algebroid data. One needs the geometric structure to be compatible with the groupoid multiplication. We will consider here only the case of differential forms and refer to \cite{BD19,CMS20} for other geometric structures.

Recall that a differential form $\Omega\in\Omega^k(\G)$ is called \emph{multiplicative} if:
\[ \gm^*\Omega=\pr_1^*\Omega+\pr^*_2\Omega. \]
We recall the following basic properties of a multiplicative form (see \cite{BC12}):
\begin{enumerate}[(i)]
\item $\Omega$ vanishing on vectors tangent to the identity section: $\gu^*\Omega=0$;
\item $\Omega$ is anti-invariant under inversion: $\ginv^*\Omega=-\Omega$;
\item The $\t$-fibers and $\s$-fibers are $\Omega$-orthogonal: 
\[ \Omega(X,Y)=0 \quad \text{ if }\quad \d_g \t(X)=\d_g \s(Y)=0; \]
\item $i_{\overleftarrow{\al}}\Omega$ is a left-invariant form, for any left-invariant vector field $\overleftarrow{\al}$. Similarly, with left-invariant replaced by right-invariant.
\end{enumerate}
The last property has the following more precise version. Consider the vector bundle map:
\[ \sigma_{_\Omega}:A\to \wedge^{k-1}T^*M,\quad \al\mapsto -\gu^*(i_\al\Omega). \]
Then one finds that for any $\al\in\Gamma(A)$
\begin{equation} 
\label{eq:multipl:left:invariant:form}
i_{\overleftarrow{\al}}\Omega=-\s^*\sigma_{_\Omega}(\al),\qquad i_{\overrightarrow{\ginv_*\al}}\Omega=\t^*\sigma_{_\Omega}(\al).
\end{equation}
The relevance of the map $\sigma_{_\Omega}$ is also shown by the following result, which is essentially due to \cite{CMS20}:

\begin{theorem}
\label{thm:local:formula}
Let $\G\tto M$ be a Lie groupoid with Lie algebroid $A$ and fix an $A$-connection $\nabla$. If $\Omega\in\Omega^k(\G)$  is a multiplicative form then:
\begin{equation}
\label{eq:pullback:mult:form}
(\exp_\nabla)^*\Omega:=-\int_0^1 (\phi^t_{X_\nabla})^*\big(\d\sigma_{_\Omega}^*\theta_{k-1}+\sigma_{_{\d\Omega}}^*\theta_{k} \big)\, \d t,
\end{equation}
where $\theta_k\in\Omega^k(\wedge^{k}T^*M)$ denotes the tautological form.
\end{theorem}

\begin{remark}
\label{rem:tautological:form}
The \emph{tautological form} $\theta_k\in\Omega^k(\wedge^{k}T^*M)$ is defined for $\al\in \wedge^{k}T^*M$ and $v_1,\dots,v_k\in T_{\al}(\wedge^{k}T^*M)$ by:
\[ (\theta_k)_\al(v_1,\dots,v_k)=\al(\d_\al p(v_1),\dots,\d_\al p(v_k)), \]
where $p:\wedge^{k}T^*M\to M$ denotes the bundle projection. This form generalizes the Liouville 1-form and, similarly to the latter, is characterized by the property that:
\[ \al^*\theta_k=\al, \quad \forall\, \al\in\Omega^k(M). \]
\end{remark}

\begin{proof}
We claim that \eqref{eq:pullback:mult:form} is equivalent to:
\begin{equation}
\label{eq:pullback:mult:form:2}
\frac{\d}{\d t}(\exp_\nabla^t)^*\Omega=-(\phi^t_{X_\nabla})^*\big(\d\sigma_{_\Omega}^*\theta_{k-1}+\sigma_{_{\d\Omega}}^*\theta_{k} \big).
\end{equation}
Clearly, if this equation holds, then integrating in $t$ and using that $\gu^*\Omega=0$, we obtain \eqref{eq:pullback:mult:form}. For the converse, if \eqref{eq:pullback:mult:form} holds, we find that:
\begin{align*} 
-(&\exp_\nabla^t)^*\Omega=-m_t^*(\exp_\nabla)^*\Omega
=m_t^*\int_0^1 (\phi^s_{X_\nabla})^*\big(\d\sigma_{_\Omega}^*\theta_{k-1}+\sigma_{_{\d\Omega}}^*\theta_{k} \big) \d s\\
&=\int_0^1 (\phi^s_{X_\nabla}\circ m_t)^*\big(\d\sigma_{_\Omega}^*\theta_{k-1}+\sigma_{_{\d\Omega}}^*\theta_{k} \big) \d s\\
&=\int_0^1 (m_t\circ\phi_{X_\nabla}^{ts})^*\big(\d\sigma_{_\Omega}^*\theta_{k-1}+\sigma_{_{\d\Omega}}^*\theta_{k} \big) \d s\\
&=\int_0^1 (\phi_{X_\nabla}^{ts})^*m_t^*\big(\d\sigma_{_\Omega}^*\theta_{k-1}+\sigma_{_{\d\Omega}}^*\theta_{k} \big) \d s\\
&=\int_0^1 (\phi_{X_\nabla}^{ts})^*\big(\d\sigma_{_\Omega}^*\theta_{k-1}+\sigma_{_{\d\Omega}}^*\theta_{k} \big) \, t\, \d s=\int_0^t (\phi_{X_\nabla}^t)^*\big(\d\sigma_{_\Omega}^*\theta_{k-1}+\sigma_{_{\d\Omega}}^*\theta_{k} \big) \d t.
\end{align*}
where we used relation \eqref{eq:flow_spray} for the flow. Differentiating both sides, we recover \eqref{eq:pullback:mult:form:2}.

Denoting by $\widetilde{X}_\nabla$ the geodesic spray of the (partial) connections on the $\t$-fibers and using the commutative diagram \eqref{eq:diagram:flows:exponential}, we can compute the left-hand side of \eqref{eq:pullback:mult:form:2} as follows:
\begin{align*} 
\frac{\d}{\d t}(\exp_\nabla^t)^*\Omega&
=\frac{\d}{\d t}(\phi^t_{\widetilde{X}_\nabla})^*\pr^*\Omega\\
&=(\phi^t_{\widetilde{X}_\nabla})^*\Lie_{\widetilde{X}_\nabla}\pr^*\Omega
=(\phi^t_{\widetilde{X}_\nabla})^*\big(\d i_{\widetilde{X}_\nabla}\pr^*\Omega+i_{\widetilde{X}_\nabla}\pr^*\d\Omega\big)
\end{align*}
On the other hand, for the right-hand side of \eqref{eq:pullback:mult:form:2} we find using the same diagram \eqref{eq:diagram:flows:exponential}:
\begin{align*}
-(\phi^t_{X_\nabla})^*\big(\d\sigma_{_\Omega}^*\theta_{k-1}+\sigma_{_{\d\Omega}}^*\theta_{k} \big)&=
-(\phi^t_{\widetilde{X}_\nabla})^*(\theta^\G)^*\big(\d\sigma_{_\Omega}^*\theta_{k-1}+\sigma_{_{\d\Omega}}^*\theta_{k} \big)\\
&=-(\phi^t_{\widetilde{X}_\nabla})^*\big(\d(\theta^\G)^*\sigma_{_\Omega}^*\theta_{k-1}+(\theta^\G)^*\sigma_{_{\d\Omega}}^*\theta_{k} \big).
\end{align*}
We conclude that \eqref{eq:pullback:mult:form:2}  holds if one can show that for any multiplicative form $\Omega$ one has on $\ker\d\t$ the equality of forms:
\begin{equation}
\label{eq:pullback:mult:form:3}
i_{\widetilde{X}_\nabla}\pr^*\Omega=-(\theta^\G)^*\sigma_{_\Omega}^*\theta_{k-1}.
\end{equation}
For this we observe that for any left-invariant vector field $\overleftarrow{\al}$ and $g\in\G$ we have:
\begin{align*} 
\big(i_{\widetilde{X}_\nabla}\pr^*\Omega&+(\theta^\G)^*\sigma_{_\Omega}^*\theta_{k-1}\big)\big|_{\overleftarrow{\al}_g}(v_1,\dots,v_{k-1})=\\
&=\Omega_g\big(\d_{\overleftarrow{\al}_g}\pr(\widetilde{X}_\nabla),\d_{\overleftarrow{\al}_g}\pr(v_1),\dots,\d_{\overleftarrow{\al}_g}\pr(v_{k-1})\big)+\\
&\hskip 1 in +\theta_{k-1}\big(\d_{\overleftarrow{\al}_g}(\sigma_{_\Omega}\circ\theta^\G)(v_1),\dots,\d_{\overleftarrow{\al}_g}(\sigma_{_\Omega}\circ\theta^\G)(v_{k-1})\big)\\
&=\Omega_g({\overleftarrow{\al}_g},\d_{\overleftarrow{\al}_g}\pr(v_1),\dots,\d_{\overleftarrow{\al}_g}\pr(v_{k-1}))+\\
&\hskip 1 in +\sigma_{_\Omega}(\al)\big(\d_g p(\d_{\overleftarrow{\al}_g}\theta^\G(v_1)),\dots,\d_g p(\d_{\overleftarrow{\al}_g}\theta^\G(v_{k-1}))\big)\\
&=(\d_{\overleftarrow{\al}_g}\pr)^*i_{\overleftarrow{\al}_g}\Omega(v_1,\dots,v_{k-1})+\\
&\hskip 1 in +\sigma_{_\Omega}(\al)\big(\d_g\s\circ\d\pr_{\overleftarrow{\al}_g}(v_1),\dots,\d_g\s\circ\d\pr_{\overleftarrow{\al}_g}(v_{k-1})\big)\\
&=(\d_{\overleftarrow{\al}_g}\pr)^*(i_{\overleftarrow{\al}_g}\Omega+\s^*\sigma_{_\Omega}(\al))(v_1,\dots,v_{k-1}),
\end{align*} 
where we used the definition of the tautological form -- see Remark \ref{rem:tautological:form} -- and the fact that for the Maurer-Cartan form one has $p\circ\theta^\G=\s\circ\pr$  -- see diagram \eqref{eq:MC:bundle:map}. Finally, using \eqref{eq:multipl:left:invariant:form}, we conclude that the last expression vanishes, so \eqref{eq:pullback:mult:form:3} holds, and the proof is complete.
\end{proof}

\subsection{Multiplicative forms as groupoid 1-cocycles}
A very useful approach to multiplicative forms, first proposed in \cite{BC12}, is to viewed them as groupoid morphisms into the abelian group $(\R,+)$, i.e., as groupoid 1-cocycles. 

First, recall that given a Lie groupoid $\G\tto M$ one can apply the tangent functor resulting in the tangent Lie groupoid $T\G\tto TM$. Less obvious, one can also form direct sums of the tangent groupoid, obtaining for each $k$ a Lie groupoid:
\[ \oplus^k_\G T\G\tto \oplus^k_M TM,\]
with source, target, and unit given by:
\begin{align*} 
\oplus^k\d\s(v_1,\dots,v_k)&=(\d\s(v_1),\dots,\d\s(v_k)),\\ 
\oplus^k\d\t(v_1,\dots,v_k)&=(\d\t(v_1),\dots,\d\t(v_k)),\\ 
\oplus^k\d\gu(w_1,\dots,w_k)&=(\d\gu(w_1),\dots,\d\gu(w_k)),
\end{align*}
and multiplication defined by:
\[ \oplus^k\d\gm((v_1,\dots,v_k),(v'_1,\dots,v'_k)=(\d\gm(v_1,v'_1),\dots,\d\gm(v_k,v'_k)). \]

Next, a differential form $\Omega\in\Omega^k(\G)$ can be viewed as a map:
\[ \overline{\Omega}:\oplus^k_\G T\G\to \R,\quad (v_1,\dots,v_k)\mapsto \Omega(v_1,\dots,v_k), \]
and one checks immediately that the multiplicativity condition for $\Omega$ amounts to:
\[ \overline{\Omega}(\oplus^k\d\gm((v_1,\dots,v_k),(v'_1,\dots,v'_k)))=
\overline{\Omega}(v_1,\dots,v_k)+\overline{\Omega}(v'_1,\dots,v'_k). \]
In other words, $\Omega$ is multiplicative if and only if $\overline{\Omega}$ is a groupoid 1-cocycle for $\oplus^k_\G T\G$.

Naturally, one would like to know what is the induced Lie algebroid 1-cocycle. For this, recall that applying the tangent functor to a Lie algebroid $p:A\to M$ also gives a tangent Lie algebroid. The bundle is
\[ \d p:TA\to TM, \]
while the anchor is given by:
\[ \d\rho:TA\to T(TM). \]
The Lie bracket can be described as follows. First, one notes that a section $\al\in \Gamma(A)$ induces two different types of sections of $TA$:
\begin{enumerate}[(a)]
\item A \emph{linear} section $\d\al:TM\to TA$;
\item A \emph{core} section $\widehat{\al}:TM\to TA$:
\[ \widehat{\al}(v_x):=v_x+\al(x)\in T_{x}M\oplus A_x\simeq T_{0_x}A.\]
\end{enumerate}
These sections generate all other sections of $TA\to TM$. One defines the Lie bracket by requiring that for any sections $\al,\be\in\Gamma(A)$:
\[  [\d\al,\d\be]_{TA}=\d[\al,\be]_A,\qquad [\d\al,\widehat{\be}]_{TA}=\widehat{[\al,\be]_{A}}, \qquad [\widehat{\al},\widehat{\be}]_{TA}=0. \]

One checks easily that given a Lie groupoid $\G\tto M$, the tangent groupoid $T\G\tto TM$ has Lie algebroid the tangent algebroid, i.e., one has:
\[ \LieA(T\G)=T\LieA(\G). \]

One can also form direct sums $\oplus^k_A TA$ of the tangent Lie algebroid $TA\to TM$, obtaining for each $k$ a Lie algebroid with vector bundle:
\[ \oplus^k_A TA\to  \oplus^k_MTM, \]
and with anchor and bracket defined componentwise. The Lie functor commutes with taking direct sums, so that:
\[ \LieA(\oplus^kT\G)=\oplus^k\LieA(T\G)=\oplus^kT\LieA(\G). \]
It follows from the results of \cite{BC12} that:

\begin{proposition}
\label{prop:multiplicative:form:bundle:map}
For $k\ge 1$, a form $\Omega\in\Omega^k(\G)$ is multiplicative if and only if the map $\overline{\Omega}:\oplus^k_\G T\G\to \R$ is a groupoid 1-cocycle. In this case, the induced Lie algebroid 1-cocycle $\LieA(\overline{\Omega})$ is the map 
\[\overline{\omega}:\oplus^k_A TA\to \R \] 
that corresponds to the $k$-form on the total space of the bundle $A$ given by:
\begin{equation}
\label{eq:multi:infinitesimal:multi:form}
\omega:=-\big(\d\sigma_{_\Omega}^*\theta_{k-1}+\sigma_{_{\d\Omega}}^*\theta_k\big)\in\Omega^k(A).
\end{equation}
\end{proposition}

\begin{proof}
The sections of the bundle $\oplus^k_A TA\to \oplus^k_M TM$ are generated by sections of the form:
\[ (\widehat{\al})^k_i:=\underbrace{0\oplus\cdots\oplus\widehat{\al}\oplus\cdots\oplus 0}_{\text{$k$ factors with $\widehat{\al}$ in $i$-entry}} \qquad
(\d\al)^k= \underbrace{\d\al\oplus\cdots\oplus\d\al}_k. \]
One then checks by direct computation using the definitions and \eqref{eq:multi:infinitesimal:multi:form} that:
\begin{align*}
\LieA\overline{\Omega}((\widehat{\al})^k_i)&=(-1)^{i+1}\pr_i^*(i_\al\Omega)=-(-1)^{i+1}\pr^*_i\sigma_{_\Omega}(\al)=\overline{\omega}((\widehat{\al})^k_i), \\
\LieA\overline{\Omega}((\d\al)^k)&=(\d i_\al\Omega+i_\al\d\Omega)=-(\d\sigma(\al)+\sigma_{_{\d\Omega}}(\al))=\overline{\omega}((\d\al)^k),
\end{align*}
where:
\[ \pr_i:\oplus^k_MTM\to \oplus^{k-1}_M TM, \quad (v_1,\dots,v_k)\mapsto (v_1,\dots,\widehat{v}_i,\dots,v_k). \]
\end{proof}

\begin{remark}
Not every k-form $\omega\in\Omega^k(A)$ induces a vector bundle map:
\[
\xymatrix{
\oplus^k_A TA\ar[d] \ar[r]^---{\overline{\omega}} & \R\ar[d] \\
\oplus^k_M TM\ar[r] & 0}
\]
When this is the case, one calls $\omega$ a \emph{linear $k$-form}. It is not hard to check that a form is linear if and only if:
\[ m_t^*\omega=t\, \omega, \quad \forall\, t>0, \]
where $m_t:A\to A$ is fiberwise multiplication by $t$. Moreover, linear $k$-forms are precisely those forms $\omega\in\Omega^k(A)$ that can be expressed as
\[ \omega=-\big(\d\sigma^*\theta_{k-1}+\nu^*\theta_k\big), \]
where $(\sigma,\nu):A\to \wedge^{k-1}T^*M\oplus \wedge^k T^*M$ are (unique) bundle maps. This can be seen, e.g., by working on a chart for $M$ over which the vector bundle $A\to M$ trivializes. We refer to  \cite{BC12} for more details on this correspondence.
\end{remark}

\subsection{Infinitesimal multiplicative forms}
Theorem \ref{thm:local:formula} and Proposition \ref{prop:multiplicative:form:bundle:map} show that for a multiplicative form $\Omega\in\Omega^k(\G)$ its local behavior around the identity section is controlled by the maps:
\begin{align*}
\sigma:=\sigma_{_\Omega}:A\to \wedge^{k-1}T^*M,\quad \al\mapsto -\gu^*(i_\al\Omega)\\
\nu:=\sigma_{_{\d\Omega}}:A\to \wedge^{k}T^*M,\quad \al\mapsto -\gu^*(i_\al\d\Omega)
\end{align*}
%
One can show that these maps are related to the Lie algebroid structure as follows:

\begin{proposition}[\cite{BC12}]
The pair $(\sigma,\nu):A\to \wedge^{k-1}T^*M\oplus \wedge^k T^*M$ satisfies:
\begin{align}
&i_{\rho(\be)}\sigma(\al)=-i_{\rho(\al)}\sigma(\be),\label{IM:0} \tag{IM0}\\
&\sigma([\al,\be])=\Lie_{\rho(\al)}\sigma(\be)-i_{\rho(\be)}\d\sigma(\al)-i_{\rho(\be)}\nu(\al), \label{IM:1}\tag{IM1}\\
&\nu([\al,\be])=\Lie_{\rho(\al)}\nu(\be)-i_{\rho(\be)}\d\nu(\al).\label{IM:2}\tag{IM2}
\end{align}
\end{proposition}

\begin{proof}
Relation \eqref{IM:0} follows immediately from the definitions. To prove \eqref{IM:1}, we use the identity:
\begin{equation}
\label{eq:general:id}
i_X i_Y\d=i_{[X,Y]}+\Lie_Y i_X -\Lie_X i_Y+\d i_X i_Y,
\end{equation}
to obtain:
\begin{align*}
\s^*(i_{\rho(\be)}\nu(\al))&=-i_{\overleftarrow{\al}}i_{\overleftarrow{\be}}\d\Omega=-\big(i_{[{\overleftarrow{\al}},{\overleftarrow{\be}}]}\Omega+\Lie_{\overleftarrow{\be}} i_{\overleftarrow{\al}} \Omega-\Lie_{\overleftarrow{\al}} i_{\overleftarrow{\be}}\Omega+\d i_{\overleftarrow{\al}} i_{\overleftarrow{\be}}\Omega\big),\\
&=\s^*\sigma([\al,\be])+\Lie_{\overleftarrow{\be}}( \s^*\sigma(\al))-\Lie_{\overleftarrow{\al}}( \s^*\sigma(\be))+\d i_{\overleftarrow{\al}} \s^*\sigma(\be)\\
&=\s^*\big(\sigma([\al,\be])-\Lie_{\rho(\al)}(\sigma(\be))+\Lie_{\rho(\be)}(\sigma(\al))+\d (\sigma(\be)(\rho(\al)))\big)\\
&=\s^*\big(\sigma([\al,\be])-\Lie_{\rho(\al)}(\sigma(\be))-i_{\rho(\be)}\d\sigma(\al)\big),
\end{align*}
where we used \eqref{eq:multipl:left:invariant:form} to pass from the first to second line, that $\overleftarrow{\al}$ is $\s$-related to $\rho(\al)$ to pass from the second to third line, and \eqref{IM:0} and Cartan's magic formula, to pass to the last line.

An entirely similar computation applied to $\d\Omega$ shows that \eqref{IM:1} also holds:
\begin{align*}
0&=-i_{\overleftarrow{\al}}i_{\overleftarrow{\be}}\d^2\Omega=-\big(i_{[{\overleftarrow{\al}},{\overleftarrow{\be}}]}\d\Omega+\Lie_{\overleftarrow{\be}} i_{\overleftarrow{\al}} \d\Omega-\Lie_{\overleftarrow{\al}} i_{\overleftarrow{\be}}\d\Omega+\d i_{\overleftarrow{\al}} i_{\overleftarrow{\be}}\d\Omega\big),\\
&=\s^*\nu([\al,\be])+\Lie_{\overleftarrow{\be}}( \s^*\nu(\al))-\Lie_{\overleftarrow{\al}}( \s^*\nu(\be))+\d i_{\overleftarrow{\al}} \s^*\nu(\be)\\
&=\s^*\big(\nu([\al,\be])-\Lie_{\rho(\al)}(\nu(\be))+\Lie_{\rho(\be)}(\nu(\al))+\d (\nu(\be)(\rho(\al)))\big)\\
&=\s^*\big(\nu([\al,\be])-\Lie_{\rho(\al)}(\nu(\be))-i_{\rho(\be)}\d\nu(\al)\big),
\end{align*}
\end{proof}

Equations \eqref{IM:0}-\eqref{IM:2} may not seem very enlightening. The reason is that, as shown by Proposition \ref{prop:multiplicative:form:bundle:map}, the more natural infinitesimal object associated with a multiplicative form $\Omega$ is the Lie algebroid 1-cocycle $\LieA(\overline{\Omega}):\oplus^k_A TA\to \R$, while those equations are formulated in terms of the maps $\sigma=\sigma_\Omega$ and $\nu=\sigma_{_{\d\Omega}}$. In fact, one has:

\begin{proposition}[\cite{BC12}]
A pair of bundle maps $(\sigma,\nu):A\to \wedge^{k-1}T^*M\oplus \wedge^k T^*M$ satisfies \eqref{IM:0}-\eqref{IM:2} if and only if the linear $k$-form $\omega\in\Omega^k(A)$ given by:
\[ \omega:=-\big(\d\sigma^*\theta_{k-1}+\nu^*\theta_k\big)\in\Omega^k(A). \]
defines a Lie algebroid 1-cocycle $\overline{\omega}:\oplus^k_A TA\to \R$. 
\end{proposition}

\begin{proof}
For the proof one needs to check that conditions \eqref{IM:0}-\eqref{IM:2} are equivalent to the Lie algebroid 1-cocycle condition:
\[ c(X,Y):=\Lie_{\rho(X)}(\overline{\omega}(Y))-\Lie_{\rho(Y)}(\overline{\omega}(X))-\overline{\omega}([X,Y])=0, \]
where $X,Y$ are any sections of the Lie algebroid $\oplus^k_A TA\to\oplus^k_M TM$. As in the proof of Proposition \ref{prop:multiplicative:form:bundle:map} we use the fact that the sections of this bundle are generated by sections of the form $(\d\al)^k$ and $(\widehat{\al})^k_i$, $(i=1,\dots,k)$. Then one finds by a direct computation that:
\begin{align*}
c\big((\widehat{\al})^k_i,(\widehat{\be})^k_j\big)&=
\left\{
\begin{array}{ll}
\pm\pr^*_{i,j}\big(i_{\rho(\be)}\sigma(\al)+i_{\rho(\al)}\sigma(\be)\big)  & \text{ if }i\not=j   \\
\\
0  & \text{ if }i= j 
\end{array}
\right.\\
c\big((\d\al)^k,(\widehat{\be})^k_i\big)&=\pm\pr_i^*\big(\sigma([\al,\be])-\Lie_{\rho(\al)}\sigma(\be)+i_{\rho(\be)}\d\sigma(\al)+i_{\rho(\be)}\nu(\al)\big)\\
c\big((\d\al)^k,(\d\be)^k\big)&=\pm\big(\nu([\al,\be])-\Lie_{\rho(\al)}\nu(\be)+i_{\rho(\be)}\d\nu(\al)\big)
\end{align*}
where:
\begin{align*}
\pr_i:\oplus^k_MTM\to \oplus^{k-1}_M TM, \quad &(v_1,\dots,v_k)\mapsto (v_1,\dots,\widehat{v}_i,\dots,v_k),\\
\pr_{i,j}:\oplus^k_MTM\to \oplus^{k-2}_M TM, \quad &(v_1,\dots,v_k)\mapsto (v_1,\dots,\widehat{v}_i,\dots,\widehat{v}_j,\dots,v_k).
\end{align*}
Hence,  $\overline{\omega}:\oplus^k_A TA\to \R$ is a Lie algebroid 1-cocycle iff \eqref{IM:0}-\eqref{IM:2} hold.
\end{proof}

This leads to the following definition:

\begin{definition}
An {\bf infinitesimal multiplicative $k$-form}, or simply an {\bf IM $k$-form}, on a Lie algebroid $A\to M$ is a pair of bundle maps 
\[ (\sigma,\nu):A\to \wedge^{k-1}T^*M\oplus \wedge^k T^*M \]
satisfying \eqref{IM:0}-\eqref{IM:2}. An IM form is {\bf closed} if $\nu=0$.
\end{definition}

Applying Lie's 2nd Theorem, one concludes that:

\begin{theorem}[\cite{BC12}]
\label{thm:linear:form:IM:forms}
Let $\G\tto M$ be a Lie groupoid with Lie algebroid $(A,[\cdot,\cdot],\rho)$ and assume that the $\t$-fibers are 1-connected. There are 1-to-1 correspondences:
 \[
\xymatrix@R=15pt@C=1pt{
& \left\{\txt{multiplicative forms\\ $\Omega\in\Omega^k(\G)$\\ \,}\right\} \ar[ld]\ar[rd]& \\
\left\{\txt{linear forms $\omega\in\Omega^k(A)$ s.t.\\ 
$\overline{\omega}:\oplus_A^k TA\to\R$\\is a Lie algebroid 1-cocycle\\ 
 \,} \right\}\ar[rr]\ar[ru] & &
\left\{\txt{IM-forms $(\sigma,\nu)$\\ $\sigma:A\to \wedge^{k-1}T^*M$\\ $\nu:A\to \wedge^k T^*M$\\ \,} \right\}\ar[lu]\ar[ll]
}
\]
\end{theorem}

\begin{remark}
Notice that under the correspondences of the theorem \emph{closed} multiplicative forms $\Omega\in\Omega^k(\G)$ correspond to:
\begin{enumerate}
\item[-] \emph{exact linear forms} $\omega=\d\al\in\Omega^k(A)$, i.e., linear multiplicative forms which are exact with a primitive $\al\in\Omega^{k-1}(A)$ which is also a linear form;
\item[-] \emph{closed IM-forms} $\sigma:A\to \wedge^{k-1}T^*M$, i.e., IM-forms $(\sigma,\nu)$ with $\nu\equiv 0$.
\end{enumerate}
When $k=2$, one also checks that \emph{non-degenerate} multiplicative 2-forms $\Omega\in\Omega^2(\G)$ correspond to non-degenerate linear 2-forms $\omega\in\Omega^2(A)$ and to IM 2-forms $(\sigma,\nu)$ for which $\sigma:A\to T^*M$ is an isomorphism. The case $k=2$ was first discussed in \cite{BCWZ04} and developed further in \cite{BCO09}.
\end{remark}

\section{From Lie algebroids to local Lie groupoids}
\label{sec:local:integration:groupoids}

\subsection{Local Lie groupoids}
\label{sec:local:groupoid}
There are different possible choices of axioms for a local Lie groupoid and these choices have important consequences, as we will see shortly. We will use as main reference for local Lie groupoids \cite{FM20}. 
In this paper we adopt the following definition, which in the language of \cite{FM20} corresponds to a 3-associative local Lie groupoid. 

\begin{definition}
    A \textbf{local Lie groupoid} $G$ over a manifold $M$ is a manifold $G$,
    together with maps:
    \begin{enumerate}
        \item[-] {\bf source/target:} $\s, \t : G \to M$ submersions;
        \item[-] {\bf units:} $u : M \to G$ a smooth map;
        \item[-] {\bf multiplication:} $m : \U \to G$ a submersion, where $\U\subset  G\timesst G$ is an open neighborhood of:
            \[ (G\timesst M) \cup (M\timesst G)
                = \bigcup_{g\in G} \{ (g,u(s(g))), (u(t(g)), g) \} ;\]
        \item[-] {\bf inversion} $i : \V \to \V$ a smooth map, where $\V \subset  G$ is an open neighborhood of $u(M)$
            such that $\V\timesst \V \subset \U$;
    \end{enumerate}
    such that the following axioms hold:
    \begin{enumerate}[({A}1)]
        \item $\s(m(g,h)) = \s(h)$ and $\t(m(g,h)) = \t(g)$ for all $(g,h) \in \U$;
        \item $m(m(g,h),k) = m(g,m(h,k))$, if $(g,h),(h,k),(m(g,h),k),(g,m(h,k))\in \U$;
        \item $m(g,u(\s(g))) = m(u(\t(g)),g) = g$ for all $g\in G$;
        \item $\s(i(g)) = \t(g)$ and $\t(i(g)) = s(g)$ for all $g\in\V$;
        \item $m(i(g),g) = u(\s(g))$ and $m(g,i(g)) = u(\t(g))$ for all $g\in \V$.
    \end{enumerate}\label{def:llg}
\end{definition}

For the multiplication we usually write $gh$ instead of $m(g,h)$ and the definition requires that it is defined provided one of the arrows is small enough. Similarly, for the inversion we shall write $g^{-1}$ instead of $i(g)$, and it is defined provided $g$ is small enough. A local Lie groupoid with space of objects $M=\{*\}$ will be called a {\bf local Lie group}. Just as in the case of Lie groupoids, the space of arrows $G$ need not be Hausdorff. However, all other manifolds, including $M$, the source and target fibers, are assumed to be Hausdorff.  Occasionally, we will deal with {\bf local topological groupoids}: the definition is analogous but one works in the topological category instead. 

There are two different ways of obtaining a smaller local Lie groupoid $G'$ from a given local Lie groupoid $G$, and  both of them are relevant to us:
\begin{itemize}
    \item {\bf Restriction:} we say that $G'$ is obtained by \emph{restricting} $G$, if both local
        groupoids have the same manifolds of arrows and objects, the same
        source and target maps, and the multiplication and inversion in $G'$
        are obtained by restricting the ones of $G$ to smaller domains;
    \item {\bf Shrink:} we say that $G'$ is obtained by \emph{shrinking} $G$, if $G'$ is an
        open neighborhood of $M$ in $G$, the source and target maps are the
        restrictions of $s$ and $t$ to $G'$, multiplication is the restriction
        of $m$ to $\U'=\U \cap (G' \timesst G')\cap m^{-1}(G')$, and
        inversion is the restriction of $i$ to $\V'=(\V \cap G') \cap i(\V \cap
        G')$.
\end{itemize}

Morphisms between local Lie groupoids are defined in a more or less obvious way (see \cite{FM20}). If $G'$ is obtained from $G$ by either restricting or shrinking, the inclusion $G' \to G$ is a morphism of local Lie groupoids. 

\begin{example}[Neighborhoods of the identity of a Lie groupoid]
\label{ex:restriction}
Let $\G\tto M$ be a Lie groupoid. Any open $\U\subset \G^{(2)}$ containing
$(\G\timesst M) \cup (M\timesst \G)$ determines a restriction $G$ of $\G$. On
the other hand, any open neighborhood $G'\subset \G$ of the identity manifold
$M$ determines a local Lie groupoid $G'$ shrinking $\G$. 
\end{example}

\begin{example}[Coverings \cite{FM20,Olver:Lie}]
\label{ex:cover:groupoid}
One way of producing local Lie groupoids (even local Lie groups) which are not restrictions or shrinkings of Lie groupoids is to pass to covering spaces. For example, starting with the abelian Lie group $G=\R^2$ and removing a point distinct from the identity, $G'=\R^2-(1,0)$ one obtains a local Lie group. The universal covering space $\widetilde{G'}$ has a unique local Lie group structure making the projection $\pi:\widetilde{G'}\to G$ a morphism of local Lie groups. We will see later why the resulting local Lie group is not the restriction or shrinking of any Lie group.
\end{example}

Many constructions for Lie groupoids extend to local Lie groupoids, under appropriate assumptions. This is discussed in detail in \cite{FM20}. Here we will assume that a local Lie groupoid $G\tto M$, with domains $\U$ and $\V$ for the multiplication and inversion maps, has the following properties
\begin{enumerate}[(a)]
\item $M$ is connected;
\item $G$ is $\s$-connected and $\t$-connected;
\item For all $x\in M$, the set $\{ (g,h) \in \U \mid \s(g) = \t(h) = x \}$ is connected;
\item For every $(g,h) \in \U$, there is a path $\gamma:I\to G$ from $\t(h)$ to $g$
    such that $(\gamma(t),h) \in \U$ for all $t\in I$,
    or there is a path $\gamma:I\to G$ from $\s(g)$ to $h$ such that
    $(g,\gamma(t)) \in \U$ for all $t\in I$.
\item The left/right multiplications induce isomorphisms: 
\[ \d_h L_g:T^\t_{h} G\to T^\t_{gh}{G},\qquad \d_g R_h:T^\s_{g} G\to T^\s_{gh}{G}, \]
whenever $(g,h)\in \U$.
\end{enumerate}
These assumptions are not too strong: one can show that any Lie groupoid has a shrinking satisfying all these properties. 

Using property (d), one can define the Lie algebroid of a local Lie groupoid in the usual manner. Moreover:

\begin{proposition}
\label{prop:generate:inverses}
Let $G\tto M$ be a local Lie groupoid satisfying properties (a)-(d). Then any open $M\subset U\subset G$ generates $G$, i.e., any $g\in G$ can be factored as
\[ g=h_1(h_2(\cdots (h_{n-1}h_n))), \]
where $h_i\in U$. In particular, the $h_i$ can be chosen to be invertible.
\end{proposition}

\subsection{Local integration of Lie algebroids}
\label{sec:local:integration:grpd}

\begin{theorem}[\cite{CrainicFernandes:integrability}]
Every Lie algebroid integrates to a \emph{local} Lie groupoid. 
\end{theorem}

There are two distinct approaches to this result: the original one from \cite{CrainicFernandes:integrability} using  the $A$-path space, and the more recent approach of \cite{CMS20b}, which uses $A$-connections to build the so-called spray groupoid.
 
\subsubsection{$A$-path local integration}
We let $I=[0,1]$. Given an $A$-path $a:I\to A$ we denote by $\gamma_a:=p\circ a:I\to M$ its base path. The set of $A$-paths:
\[ P(A):=\left\{a:I\to A~|~ \text{$a$ is $A$-path of class $C^1$, with $\gamma_a$ of class $C^2$}\right\} \]
is a Banach manifold, which we call the \emph{space of $A$-paths}.  

An $A$-path can be thought of as a Lie algebroid morphism 
\[ a\d t:TI\to A. \]
Then one can define an $A$-path homotopy between $A$-paths $a_1:TI\to A$ and $a_2:TI\to A$ to be a Lie algebroid morphism:
\[ \Phi: T(I\times I)\to A, \]
satisfying the boundary conditions:
\[ \Phi|_{TI\times\{0\}}=a_0,\quad \Phi|_{TI\times\{1\}}=a_1,\quad \Phi|_{\{0\}\times TI}=\Phi|_{\{1\}\times TI}=0. \]

Fix an $A$-connection $\nabla$ on $A$. The map that associates to an element $a_0\in A$ the geodesic $a:I\to A$ with initial condition $a(0)=a_0$ gives an exponential map:
\[ \exp_\nabla:A\to P(A). \] 
We use the same notation as we did before for the groupoid version of the exponential map. From the context it should be clear which one we refer to. If $\sim$ denotes the equivalence relation on $P(A)$ determined by $A$-homotopy, then we have the $A$-homotopy groupoid, as mentioned in the introduction:
\[ \G(A)=P(A)/\sim. \]
The multiplication is induced by concatenation of $A$-paths -- see \cite{CrainicFernandes:integrability} for details. The exponential map then induces a map into the quotient $\exp_\nabla:A\to \G(A)$, which when $\G(A)$ is smooth is the groupoid version of the exponential map. 

In general, the quotient $\G(A)$ fails to be smooth. However, it is proved in  \cite{CrainicFernandes:integrability} that for a sufficiently small neighborhood $V\subset A$ of the zero section, the quotient 
\[ G:=\exp_\nabla(V)/\sim \]
is smooth and becomes a local Lie groupoid with (partial) multiplication defined by concatenation of $A$-paths. Notice that this integration depends on the choice of connection $\nabla$ and open set $V$. However, two different choices lead to local Lie groupoids which have isomorphic shrinkings.

\subsubsection{The spray local integration}
Let us now consider the alternative method to obtain local integrations of $A$ proposed in \cite{CMS20b}. 

As we saw in Section \ref{subsection:groupoid:exponential}, given a Lie groupoid $\G\tto M$ and an $A$-connection $\nabla$ on its Lie algebroid $p:A\to M$, using the exponential map $\exp_\nabla:A\to\G$ one can pull back the Lie groupoid structure from $\G$ to a local Lie groupoid structure on a neighborhood $V\subset A$ of the zero section. The main observation is that, as we saw before, the resulting local groupoid structure is expressed purely in terms of Lie algebroid data. 

Hence, starting with \emph{any} Lie algebroid $A$ one chooses an open $V\subset A$ containing the zero section which is invariant under the map $a\mapsto -\phi^1_{X_\nabla}$ and defines all the structure maps exactly as in Section \ref{subsection:groupoid:exponential}. The main difficulty is to show that the local multiplication, given by \eqref{eq:multiplication:exponential}, is actually well-defined. We refer to \cite{CMS20b} for details.

\subsection{Local integration of IM-forms}
\label{sec:local:integration:forms}
Suppose we are given an IM-form $(\sigma,\nu):A\to \wedge^{k-1}T^*M\oplus \wedge^k T^*M$ on a Lie algebroid $A$ or, equivalently, the linear $A$-form 
\[ \omega\in\Omega^k(A),\quad \omega:=-\big(\d\sigma^*\theta_{k-1}+\nu^*\theta_k\big), \]
so that $\overline{\omega}:\oplus^k_A TA\to \R$ is a Lie algebroid 1-cocycle -- see Theorem \ref{thm:linear:form:IM:forms}. We would like to integrate it to a multiplicative $k$-form $\Omega$ on a local Lie groupoid $G\tto M$ integrating $A$. One can proceed in two distinct ways, as we now explain.

\subsubsection{Integration of 1-cocycles \cite{BC12}} 
If $G\tto M$ has 1-connected $\t$-fibers then the groupoid $\oplus^k_G TG\tto \oplus^k_MTM$ also has 1-connected $\t$-fibers. Hence, Lie's Second Theorem for local Lie groupoids allows to integrate the Lie algebroid 1-cocycle $\overline{\omega}:\oplus^k_A TA\to \R$ to a Lie groupoid 1-cocycle $\Phi:\oplus^k_G TG\to \R$. 

We claim that the map:
\[ (v_1,\dots,v_k)\mapsto \Phi(v_1,\dots,v_k), \]
is multilinear and skew-symmetric, so that $\Phi=\overline{\Omega}$ for some $\Omega\in\Omega^k(G)$. Notice that then Proposition \ref{prop:multiplicative:form:bundle:map} shows that $\Omega$ is a multiplicative form. To prove the skew-symmetry, one observes that the map:
\[ I:\oplus^k_G TG\to \oplus^k_G TG,\quad I(v_1,\dots,v_i,\dots,v_j,\dots,v_k):=(v_1,\dots,v_j,\dots,v_i,\dots,v_k), \]
is a groupoid morphism whose induced Lie algebroid morphism is:
\[ I_*:\oplus^k_A TA\to \oplus^k_A TA,\quad I_*(a_1,\dots,a_i,\dots,a_j,\dots,a_k):=(a_1,\dots,a_j,\dots,a_i,\dots,a_k). \]
Then $-\Phi\circ I:\oplus^k_G TG\to\R$ is also a groupoid 1-cocycle and the induced Lie algebroid 1-cocycle is:
\[ (-\Phi\circ I)_*=-\overline{\omega}\circ I_*=\overline{\omega}. \]
By uniqueness in Lie's Second Theorem we must have $\Phi=-\Phi\circ I$, which is precisely the skew-symmetry. For multilinearity one proceeds similarly replacing $I$ by the Lie groupoid morphisms:
\begin{align*}
\oplus^k_G TG\to \oplus^k_G TG,&\quad (v_1,\dots,\lambda v_i,\dots,v_k)\mapsto \lambda(v_1,\dots,v_i,\dots,v_k)\\
\oplus^{k+1}_G TG\to \oplus^k_G TG,&\quad (v_1,v_1',v_2,\dots,v_k)\mapsto (v_1+v_1',v_2\dots,v_k).
\end{align*}

Since $\Omega$ is a multiplicative form, by Proposition \ref{prop:multiplicative:form:bundle:map}, the Lie groupoid 1-cocycle $\overline{\Omega}$ differentiates to a Lie algebroid 1-cocycle $\overline{\omega}=(\overline{\Omega})_*:\oplus^k_ATA\to\R$ such that: 
\[ \omega=-\big(\d\sigma^*\theta_{k-1}+\nu^*\theta_k\big)=-\big(\d\sigma_{_\Omega}^*\theta_{k-1}+\sigma_{_{\d\Omega}}^*\theta_k\big). \]
The properties of the tautological forms imply that
\[\sigma=\sigma_{_{\Omega}},\qquad \nu=\sigma_{_{\d\Omega}}.\]
We conclude that $\Omega$ is a multiplicative form integrating the IM form $(\sigma,\nu)$.

\subsubsection{Integration using local formulas \cite{CMS20}} If one starts with the spray local integration $G\tto M$ associated with an $A$-connection $\nabla$, one can also use the explicit local formulas for multiplicative forms \eqref{eq:pullback:mult:form} to integrate an IM-form. 

An IM form $(\sigma,\nu):A\to \wedge^{k-1}T^*M\oplus \wedge^k T^*M$ defines a differential $k$-form $\Omega\in\Omega^k(U)$ on a sufficiently small neighborhood $U\subset A$ of the zero section by
\[ \Omega:=-\int_0^1 (\phi^t_{X_\nabla})^*\big(\d\sigma^*\theta_{k-1}+\nu^*\theta_{k} \big)\, \d t. \]
Then one needs to check that this form is multiplicative and that it induces the given IM form. To prove these it is more convenient to view both $(\sigma,\nu)$ and $\Omega$ as maps $\overline{\omega},\overline{\Omega}:\oplus^k_A TA\to\R$. Then:
\begin{equation}
\label{eq:local:formula:aux} 
\overline{\Omega}=\int_0^1 (\phi^t_{X_\nabla})^*\overline{\omega}\, \d t.
\end{equation}
We claim that this last formula is just the standard formula for integrating Lie algebroid 1-cocycles to Lie groupoid 1-cocycles, so the result follows. 

For that let $G\tto M$ be a local Lie groupoid with 1-connected $\t$-fibers and denote by $A\to M$ its Lie algebroid. Then given a Lie algebroid 1-cocycle $c:A\to\R$ the corresponding Lie groupod 1-cocycle $C:G\to\R$ is given by:
\begin{equation}
\label{eq:integrating:morphisms} 
C(g)=\int_0^1 c\big(a(t)\big)\, \d t,
\end{equation}
where $a:I\to A$ is the $A$-path obtained by applying the Maurer-Cartan form:
\[ a(t)=\theta^G(\dot{g}(t))\quad (t\in I). \]
to any path $g:I\to G$ in the $\t$-fiber, with $g(0)=1_{\t(g)}$ and $g(1)=g$. 

Let $G\tto M$ be the spray groupoid $V\tto M$ associated with a connection $\nabla$, the discussion after Lemma \ref{lem:MC:form} shows that the exponential map of $\nabla$ is just the identity map $\exp_\nabla(a)=a$. The  
$A$-path $t\mapsto \phi^t_{X_\nabla}(a_0)$ represents the path $g:I\to V$ in the spray groupoid which lies in the $\t$-fiber connecting the identity $1_{x_0}$ with $g=\exp_\nabla(a_0)=a_0$. It is now easy to check that \eqref{eq:local:formula:aux} is just a special instance of \eqref{eq:integrating:morphisms} applied to the local Lie groupoid $\oplus^k_G TG\tto \oplus^k_M TM$, where $G\equiv V$ is the spray groupoid.

\section{From local Lie groupoids to global Lie groupoids}
\label{sec:completion}

\subsection{Enlarging local groupoids}
\label{sec:associativity}
The associativity axiom (A2) for a local groupoid requires that for every triple of composable arrows one has:
\[ g(hk)=(gh)k, \]
provided both sides are defined. While for a (global) Lie groupoid this implies that all higher associativities hold, this is not true for a local Lie groupoid. For example, given 4 elements $g$, $h$, $k$ and $l$, the possible products form a pentagon:
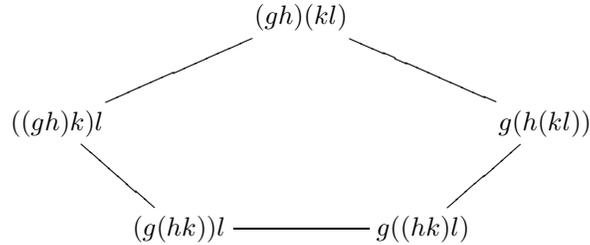
\[
\xymatrix@C=5pt{ 
& & (gh)(kl)\ar@{-}[drr]\ar@{-}[dll] \\
((gh)k)l\ar@{-}[dr] & & & & g(h(kl))\ar@{-}[dl] \\
& (g(hk))l\ar@{-}[rr] & &  g((hk)l)}
\]
Each edge represents a move that uses only the 3-associativity property. In a global groupoid all vertices of the pentagon are defined and the 3-associativity implies 4-associativity. In a local Lie groupoid it is possible that, e.g., $(gh)(kl)$ and $(g(hk))l$ are defined but none of the other vertices are defined, and then 3-associativity does not allow one to conclude that 4-associativity holds. This happens, for example, in the case of the covering groupoids from Example \ref{ex:cover:groupoid} -- see \cite{FM20,Olver:Lie} for details. 

Notice that restricting or shrinking a local groupoid can make it ``more associative''. In fact, one has the following result:

\begin{proposition}[\cite{FM20}]
Let $G$ be a local Lie groupoid. For each $n\ge 3$ there is a restriction of $G$
which is $n$-associative.
\end{proposition}
 
In the previous proposition we have \emph{fix} $n$. A different issue is whether one can find a restriction which is $n$-associative for all $n\ge 3$.

\begin{definition}
    A local Lie groupoid is called \emph{globally associative} if it is
    associative to every order $n\geq 3$.
\end{definition}

Obviously, a local groupoid $G\tto M$ obtained by restricting or shrinking a global Lie groupoid, as in Example \ref{ex:restriction}, is globally associative. For example, a local Lie \emph{group} $G$ always has a shrinking that is globally associative: if $\gg$ is the Lie algebra of $G$ then a shrinking of $G$ is isomorphic to a neighborhood of the identity of the 1-connected integration $\G(\gg)$.  However, for a local Lie \emph{groupoid} integrating a non-integrable Lie algebroid  this argument fails.

We will say that a local Lie groupoid $G\tto M$ is \emph{globalizable} or \emph{enlargeable} if it is isomorphic to a 
restriction of an open neighborhood of the unit section of a Lie groupoid $\G\tto M$. The problem of determining if a Lie group is globalizable was first studied by Mal'cev \cite{Malcev41} -- see also \cite{Olver:Lie} -- who showed that the failure of global associativity is the only obstruction to embedding a local group in a global one. In \cite{FM20}, Mal'cev's Theorem was extended to realm of local Lie groupoids:

\begin{theorem}[Mal'cev's theorem for groupoids \cite{FM20}]
    A local Lie groupoid is globalizable if and only if its multiplication is globally associative.
\end{theorem}

We will sketch a proof in the next two paragraphs. For now we observe the following consequence:

\begin{corollary}
A local Lie groupoid has a restriction which is globally associative if and only if it has integrable Lie algebroid.
\end{corollary}

\begin{proof}
If a local Lie groupoid $G$ has a restriction $G'$ which is globally associative then $G$ and $G'$ have the same Lie algebroid $A$. By the theorem, $G'$ is globalizable, so $A$ is an integrable algebroid. Conversely, if a local Lie groupoid $G$ has an integrable Lie algebroid $A$, then there is a neighborhood $G'\subset G$ of the units which is isomorphic to a neighborhood in $\G(A)$ of the units. It follows that $G'$ is globally associative.
\end{proof}

\begin{remark}
\label{rem:Bailey:Gualtieri}
The previous results show that one should be careful with the associativity axiom for a local Lie groupoid $G$. For example, in the definition used in \cite{BG16} this axiom reads as follows:
\begin{equation}
\label{eq:too:strong} 
\txt{If either $(gh)k$ or $g(hk)$\\ is defined} \quad \Longrightarrow \quad \txt{both $(gh)k$ and $g(hk)$ are defined\\ and $(gh)k=g(hk)$}.
\end{equation}
Note that this implies that $G$ is globally associative, hence has integrable Lie algebroid. So the definition in \cite{BG16} excludes, e.g., the $A$-path local integration and the spray groupoid for non-integrable Lie algebroids.
\end{remark}

\subsection{The associative completion}
\label{sec:associative:completion}
Mal'cev's Theorem is based on an important construction associated with a local groupoid, namely its  {\em associative completion}.

Given a local Lie groupoid $G\tto M$, its associative completion is a groupoid $\AsCo(G)\tto M$ obtained as follows. One introduces the set of \emph{well-formed} words on $G$ formed by words
$(w_1,\dots,w_n)$ in $G$ whose source and target match:
\[ W(G) :=\bigsqcup_{n\geq 1} \underbrace{G\timesst G\timesst \cdots \timesst G}_{\text{$n$ times}}, \]
Given a well-formed word $w=(w_1,\dots,w_k,w_{k+1},\dots,w_n)$ such that  $w_k$ and $w_{k+1}$ can be
composed, we have a new well-formed word $w' = (w_1,\dots,w_kw_{k+1},\dots,w_n)$, and we say that $w'$ is is obtained from $w$ by \emph{contraction} or that $w$ is obtained from $w'$ by \emph{expansion}.
Contractions and expansions generate an equivalence relation $\sim$ on $W(G)$,
and one defines the associative completion of $G$ to be the space of
equivalence classes:
\[ \AsCo(G) := W(G)/{\sim}.\]
Notice that $\AsCo(G)$ has a groupoid structure:
\begin{itemize}
\item The source and target maps are given, respectively, by:
\[ [(w_1,\dots,w_n)]\mapsto \s(w_n),\quad [(w_1,\dots,w_n)]\mapsto \t(w_1);\]
\item Multiplication is given by concatenation of words:
\[ [(w_1,\dots,w_n)]\cdot [(w'_1,\dots,w'_m)]=[(w_1,\dots,w_n,w'_1,\dots,w'_m)];\]
\item The unit section is the map:
\[ x\mapsto [(1_x)] ;\]
\item Inversion is given by:
\[ [(w_1,\dots,w_n)]\mapsto [(w_n^{-1},\dots,w_1^{-1})], \]
where one uses Proposition \ref{prop:generate:inverses} to find representatives for equivalence classes of the form $[(w_1,\dots,w_n)]$ where every $w_i$ is invertible. 
\end{itemize}
There is an obvious completion map $G\to \AsCo(G)$, which is a morphism of local groupoids. 

In general, if $G$ is a local Lie groupoid, $\AsCo(G)$ is a topological groupoid for the quotient topology, but may fail to be a Lie groupoid. Moreover, the associative completion is characterized by the following universal property: for every topological groupoid $\H$ and every continuous morphism of local groupoids $\phi:G\to \H$ there is a unique morphism of topological groupoids $\widehat{\phi}:\AsCo(G)\to \H$ making the following diagram commute:
\[
\xymatrix{
G\ar[d]\ar[r]^{\phi} & \H \\
\AsCo(G)\ar@{-->}[ru]_{\widehat{\phi}}
}
\]
Notice also that given a morphism of local Lie groupoids $\phi:G\to H$ there is an induced morphism of topological groupoids $\AsCo(\phi):\AsCo(G)\to \AsCo(H)$ making the following diagram commute:
\[
\xymatrix{
G\ar[d]\ar[r]^{\phi} & H\ar[d] \\
\AsCo(G)\ar[r]_{\AsCo(\phi)}& \AsCo(H)
}
\]
Hence, $\AsCo(-)$ is a functor from the category of local Lie groupoids to the category of topological groupoids.

\subsection{Associators and smoothness}
The next natural question, which is also at the core of Mal'cev's Theorem is:
\begin{itemize}
\item Given a local Lie groupoid $G$, when is $\AsCo(G)$ a Lie groupoid?
\end{itemize}
Here, of course, the smooth structure we are wondering about for $\AsCo(G)$ is the quotient smooth structure obtained from $W(G)$ -- the set of well-formed words -- which is itself a manifold with connected components of possibly different dimension.

Given a local Lie groupoid we will call an element in the isotropy 
\[ g\in G_x=\s^{-1}(x)\cap \t^{-1}(x) \] 
an \emph{associator} at $x$ if there is a well-formed word $(w_1,\dots,w_n)$ which admits two sequences of contractions: one ending at $g$ and the other one ending at the identity $1_x$. The set of all associators will be denoted $\Assoc(G)$ and is contained in the kernel of the completion map $G\to \AsCo(G)$. One of the main results in \cite{FM20} shows that the associators control the smoothness of $\AsCo(G)$:

\begin{theorem}[\cite{FM20}]
\label{thm:AC:smooth}
If $G$ is a local Lie groupoid, then $\AsCo(G)$ is smooth if and only if $\Assoc(G)$ is uniformly discrete in $G$, i.e., if there is an open $M\subset U\subset G$ such that:
\[ U\cap \Assoc(G)=M. \]
In this case, the completion map $G \to \AsCo(G)$ is a local diffeomorphism, so $G$ and $\AsCo(G)$ have the same Lie algebroid.
\end{theorem}

\begin{proof}[Sketch of the proof]
The main idea is to use local bisections of the local Lie groupoid $G\tto M$ to construct local charts for $\AsCo(G)$. Given an element $g\in\AsCo(G)$, represented by a word $(w_1,\dots,w_n)$, to construct a chart near $g$ one considers submanifolds $N_i \subset G$ of the same dimension as $M$ through $w_i$ that are transverse to both source and target fibers, i.e., local bisections of $G$ through $w_i$. Then $\s$ and $\t$ define diffeomorphisms between the $N_i$'s and open subsets of $M$. Write $\s_i =\s|_{N_i}$ and $\t_i =\t|_{N_i}$ for these diffeomorphisms.

Fix $k\in\{1,\dots,n\}$ and let $U$ be a small neighborhood of $w_k$. Then a chart near $g$ is defined as
\[ \varphi : U \to \AsCo(G) : h \mapsto [h_1,\dots,h_{k-1},h,h_{k+1},\dots,h_n] \]
where
    \[
        h_{k+1} = \t_{k+1}^{-1}(\s(h)), \ 
        h_{k+2} = \t_{k+2}^{-1}(\s(h_{k+1})),
        \dots, \ 
        h_n = \t_n^{-1}(\s(h_{n-1})),
    \]
    and
    \[
        h_{k-1} = \s_{k-1}^{-1}(\t(h)), \ 
        h_{k-2} = \s_{k-2}^{-1}(\t(h_{k-1})),
        \dots, \
        h_1 = \s_1^{-1}(\t(h_2)).
    \]
For $U$ sufficiently small, the map $\varphi$ is a well-defined local diffeomorphism that maps $w_k$ to $g$. The main point is that, eventually after choosing $U$ even smaller, the uniform discreteness ensures that $\varphi$ is injective. One then checks that the transition maps between different charts are smooth, so one obtains a possible non-Hausdorff, second countable, smooth structure. Still, the source and target fibers are Hausdorff, and $\AsCo(G)$ is a Lie groupoid. For all the details we refer to \cite{FM20}.
 \end{proof}

\begin{proof}[Proof of Mal'cev's Theorem]
If $G$ is a globally associative local Lie groupoid, then the only associators $\Assoc(G)$ are the units. Hence,  $\Assoc(G)$ is uniformly discrete and so, by the theorem, $\AsCo(G)$ is a Lie groupoid. Global associativity also implies that the completion map $G \to \AsCo(G)$ is injective, so $G$ is isomorphic to its image under the completion map, and so it is globalizable.
\end{proof}

\subsection{Associators and integrability}
\label{sec:associators}
An obvious consequence of Theorem \ref{thm:AC:smooth} is that a local Lie groupoid $G\tto M$ with uniformly discrete associators $\Assoc(G)$ must have an integrable Lie algebroid. The converse is not quite true: integrable Lie algebroids (even Lie algebras!) can have local integrations with non-discrete associators. We refer to \cite{FM20} for an example.  Still, there is a relationship between integrability of a Lie algebroid $A$ and the associators of a local integration $G\tto M$, provided the local integration is not ``too large''. 

To express this relationship precisely, let us recall that for a Lie algebroid $A$ one has \emph{monodromy groups} $\NN_x(A)$ which control the integrability of $A$. These groups arise by looking at the  isotropy groups $\G(A)_x$ of the $A$-homotopy groupoid $\G(A)\tto M$: the inclusion of the isotropy Lie algebra $\gg_x=\ker\rho_x\hookrightarrow A$ integrates to a surjective group morphism
\[ \G(\gg_x)\to \G(A)_x^0, \]
whose kernel is precisely $\NN_x(A)$. This monodromy group is a subgroup of the center and we have:
\[ \G(A)_x^0=\G(\gg_x)/\NN_x(A). \]
Clearly, if $\G(A)$ is smooth then $\G(A)_x^0$ is smooth, and the subgroup $\NN_x(A)\subset \G(\gg_x)$ must be discrete. Conversely, one has the following fundamental result:

\begin{theorem}[\cite{CrainicFernandes:integrability}]
A Lie algebroid $A\to M$ is integrable if and only if its monodromy groups are uniformly discrete, i.e., if and only if there exists an open $V\subset A$ containing the zero section such that:
\[ \exp_x(V\cap \gg_x)\cap \NN_x(A)=\{0\}, \qquad \forall\, x\in M,\]
where $\exp_x:\gg_x\to \G(\gg_x)$ denotes the exponential map.
\end{theorem}

Now, let $G\tto M$ be a local Lie groupoid integrating a Lie algebroid $A$. We will assume that $G$ has been shrunk so that:
\begin{itemize}
\item[(H1)] $G_x$ is a 1-connected local Lie group.
\end{itemize}
Under this assumption, we have a local Lie group homomorphism $G_x \to \G(\g_x)$ defined as follows. If $g\in G_x$, choose a path in $G_x$ from $x$ to $g$. Differentiate this path to get a $\g_x$-path and hence an element of $\G(\g_x)$. Since $G_x$ is simply-connected, this map is well-defined: any two paths in $G_x$ from $x$ to $g$ are homotopic, and such a homotopy induces a $\g_x$-homotopy. By shrinking $G$ further, we can also assume that:
\begin{itemize}
\item[(H2)] the map $G_x\to \G(\g_x)$ is injective.
\item[(H3)] the $\t$-fibers of $G$ are 1-connected.
\end{itemize}
Again, this allows us to construct a local homomorphism of topological groupoids $G\to \G(A)$:  if $g\in G$ has target $x$, choose a path in $\t^{-1}(x)$ from $x$ to $g$. Differentiate this path to get an $A$-path and hence an element of $\G(A)$. Since $\t^{-1}(x)$ is simply-connected, this map is well-defined, because any two paths in $\t^{-1}(x)$ from $x$ to $g$ are homotopic, and this homotopy induces an $A$-homotopy. 

Assumptions (H1)-(H3) are already sufficient to relate the associators with the monodromy groups, Namely, we have:
\begin{proposition}
    \label{prop:assoc-is-mon}
Let $G$ be a local Lie groupoid satisfying (H1)-(H3). Then under the natural map $G_x\to \G(\g_x)$ we have:
\[ \Assoc_x(G) \subset \NN_x(A). \]
\end{proposition}

\begin{proof}
Applying the functor $\AsCo(-)$ to the morphism $G\to \G(A)$ we obtain a commutative triangle:
\[
\xymatrix{G \ar[rr]\ar[dr] &&\G(A) \\
& \AsCo(G)\ar[ur]}
\]
Passing to isotropies, it follows that the map $G_x\to \G_x(A)$ takes the associators $\Assoc_x(G)$ to the unit. But the last map factors as:
\[
\xymatrix{G_x \ar[rr]\ar[dr] &&\G_x(A) \\
& \G(\g_x)\ar[ur]}
\]
Since the kernel of $\G(\g_x)\to\G_x(A)$ is the group $\NN_x(A)$, the result follows.
\end{proof}

One would like to obtain equality in the previous proposition. For that, one needs to construct a lift of the map $G\to \G(A)$ to the space of $A$-paths:
\[
\xymatrix{
 & & P(A) \ar[d] \\
G\ar@{-->}[urr]^P \ar[rr] && \G(A) 
}
\]
which is multiplicative in the following sense: if $w = (g_1,\dots,g_k)\in W(G)$ is a well-formed word on $G$, write $P(w) = P(g_k)\circ\cdots \circ P(g_1)$ where $\circ$ denotes concatenation of $A$-paths. Then if $w_1$ and $w_2$ are equivalent words we would like $P(w_1)$ and $P(w_2)$ to be $A$-homotopic paths. It is possible to further shrink the local groupoid $G$, so that (H1)-(H3) are satisfied and $G$ carries such a lift. One then says that $G$ is \emph{shrunk}. The lift is constructed with the help of a Riemannian structure on $M$ and an $A$-connection. We refer to \cite{FM20} for details.

The fundamental result connecting integrability with associators can now be stated as follows:

\begin{theorem}[\cite{FM20}]
\label{thm:associators}
Let $G\tto M$ be a shrunk local Lie groupoid with Lie algebroid $A$. For $x\in M$, consider $G_x$ as a subset of $\G(\g_x)$ using the natural map $G_x \to \G(\g_x)$. Then:
\[ \Assoc_x(G) = \NN_x(A) \cap G_x .\]
\end{theorem}

For the proof of this result and a simplicial interpretation we refer to \cite{FM20}.

\subsection{Extending local multiplicative forms to global multiplicative forms}

Having established the conditions for the associative completion $\AsCo(G)$ to be smooth, one has the natural problem of finding when a local multiplicative form on $G$ extends to a global multiplicative form on $\AsCo(G)$. Since $\AsCo(G)$ may not have 1-connected $\t$-fibers, the answer is perhaps surprising: always!

\begin{theorem}[\cite{BG16}]
\label{thm:local:global:forms}
Let $G$ be a local Lie groupoid with uniformly discrete associators and connected $\t$-fibers. Given any multiplicative form $\Omega\in\Omega^k(G)$ there exists a unique multiplicative form $\widetilde{\Omega}\in\Omega^k(\AsCo(G))$ whose pullback under the completion map $G\to\AsCo(G)$ is $\Omega$.
\end{theorem}

A version of this result for globally associative local groupoids -- see Remark \ref{rem:Bailey:Gualtieri} -- appears in the appendix of \cite{BG16}. The proof is only sketched there, so we will give a detailed proof below. 

In general, given a multiplicative form defined on an open neighborhood $U$ of the units of a Lie groupoid $\G\tto M$, it is not possible to extend it to a multiplicative form on $\G$ unless $\G$ has 1-connected $\t$-fibers. Underlying Theorem \ref{thm:local:global:forms} is the fact that the local groupoid $G$ ``captures'' the topology of $\AsCo(G)$. 

To make this more explicit, consider a Lie groupoid $\G\tto M$ with connected $\t$-fibers. By Proposition \ref{prop:generate:inverses}, any neighborhood $M\subset U\subset \G$ generates $\G$. If we consider $U$ as a local groupoid obtained by shrinking $\G$, then the universal property of $\AsCo(-)$ implies that the inclusion $i:U\hookrightarrow \G$ induces a Lie groupoid morphism:
\[ \widehat{i}:\AsCo(U)\to \G, \]
which is a \emph{surjective} local diffeomorphism.  We call $U\subset \G$ a \emph{full neighborhood} if this map is an isomorphism. Hence, Theorem \ref{thm:local:global:forms} has the following consequence:

\begin{corollary}
\label{cor:local:global:forms}
Let $\G\tto M$ be a Lie groupoid with connected $\t$-fibers. Every multiplicative form defined on a full neighborhood extends to a unique multiplicative form defined on $\G$.
\end{corollary}

The following example illustrates the difference between simple neighborhoods and full neighborhoods, already in the case of Lie groups.

\begin{example}
Consider the Lie group $\S^1:=\{z\in\C:|z|=1\}$ and for $a\in (0,\pi]$ set:
\[ U_a:=\big\{ e^{it}\in \S^1: t\in (-a,a) \big\}. \]
We claim that $U_\pi$ is a full neighborhood while $U_{\pi/2}$ is not. 

For the local Lie group structure on $U_{\pi/2}$, obtained by shrinking $\S^1$, the map:
\[ \phi:U_{\pi/2}\to (\R,+), \quad e^{it}\mapsto t, \]
is easily seen to be a morphism of local Lie groups. Hence, there is an induced Lie group morphism:
\[ \widehat{\phi}:\AsCo(U_{\pi/2})\to \R, \]
which is local diffeomorphism. Since $\AsCo(U_{\pi/2})$ is connected and $\R$ is 1-connected, we must have $\AsCo(U_{\pi/2})\simeq \R$, so that $U_{\pi/2}$ is not a full neighborhood. Notice that the local morphism $\phi$ does not extend to a Lie group morphism $\S^1\to \R$ since the only such morphism is the trivial one -- this would be the case $k=0$ in Corollary \ref{cor:local:global:forms}, so the corollary does not hold for $U_{\pi/2}$.

For the local group structure on $U_{\pi}$ the map:
\[ \phi:U_{\pi}\to (\R,+), \quad e^{it}\mapsto t, \]
is not anymore a morphism of local Lie groups: for example, $2e^{i\frac{3\pi}{4}}=e^{-i\frac{\pi}{2}}$ but $2 \phi(e^{i\frac{3\pi}{4}})=\frac{3\pi}{2}\not=-\frac{\pi}{2}=\phi(2e^{i\frac{3\pi}{4}})$. On the other hand, we can consider the inclusion:
\[  i:U_{\pi}\hookrightarrow \S^1\] 
which is obviously a local Lie group morphism. The induced Lie group morphism $\widehat{i}:\AsCo(U_{\pi})\to \S^1$
is surjective and local diffeomorphism. One checks easily that this map is injective, so $\AsCo(U_{\pi})\simeq \S^1$ and $U_\pi$ is full. 

%
\end{example}

We now turn to the proof of Theorem \ref{thm:local:global:forms}. The idea is simple: we extend  $\Omega$  to a  form on $\AsCo(G)$ by requiring the multiplicativity condition. The problem of course is that two words can represent the same element of $\AsCo(G)$, so we need some ``book keeping'', and for that we proceed as follows.

We denote the set of well-formed words of length $l$ by:
\[ G^{(l)}:=\underbrace{G\timesst \cdots \timesst G}_{\text{$l$ times}}, \]
so that $W(G):=\bigsqcup_{l=1}^{+\infty} G^{(l)}$, and we have the quotient map:
\[ \Phi:W(G)\to \AsCo(G),\quad \Phi(g_1,\dots,g_l):=[(g_1,\dots,g_l)]. \]
We also introduce face maps for $1\le i \le l-1$:
\[ \partial_i: \U^{(l)}_i\to G^{(l-1)},\quad \partial_i(g_1,\dots,g_i,g_{i+1},\dots,g_l)=(g_1,\dots,g_ig_{i+1},\dots,g_l), \]
defined on the open subset $\U^{(l)}_i\subset G^{(l)}$ given by: 
\[ \U^{(l)}_i=\{(g_1,\dots,g_i,g_{i+1},\dots,g_l)\in G^{(l)}: (g_i,g_{i+1})\in\U\}. \]
The face maps are submersions and a word $w'$ is a contraction of the word $w$ if for some $i$ one has:
\[ w'=\partial_i w. \]
These contractions/expansions then generate the equivalence $\sim$, and given words $w_1,w_2\in W(G)$ one has $\Phi(w_1)=\Phi(w_2)$ if and only if $w_1\sim w_2$.

After these preliminaries, consider the family of forms $\Omega^{(l)}\in\Omega^k(G^{(l)})$ defined by:
\begin{equation}
\label{eq:stage:multipl} 
\Omega^{(l)}:=\sum_{j=1}^l \pr_j^*\Omega,
\end{equation}
where $\pr_i:G^{(l)}\to G$ denotes the projection in the factor $i$. Using the multiplicativity of $\Omega$ we find:

\begin{lemma}
The forms $\Omega^{(l)}\in\Omega^k(G^{(l)})$ satisfy:
\begin{equation}
\label{eq:face:multipl} 
\partial_i^*\Omega^{(l-1)}=\Omega^{(l)}, \quad (1\le i\le l-1),
\end{equation}
on the open subset of $\U^{(l)}_i$ where $\partial_i$ is defined.
\end{lemma}

\begin{proof}
Notice that:
\[
\pr_j\circ\, \partial_i=
\left\{
\begin{array}{lr}
\pr_j, & \textrm{for }j<i,   \\
\gm\circ(\pr_i\times\pr_{i+1}),&\textrm{for }j=i,  \\
\pr_{j+1},& \textrm{for }j>i.
\end{array}
\right.
\]
So using the multiplicativity of $\Omega$ we find:
\begin{align*} 
\partial_i^*\Omega^{(l-1)} 
&=\sum_{j<i} \pr_j^*\Omega+(\pr_i\times\pr_{i+1})^*\gm^*\Omega+\sum_{j>i} \pr_{j+1}^*\Omega\\
&=\sum_{j<i} \pr_j^*\Omega+(\pr_i\times\pr_{i+1})^*(\pr_1^*\Omega+\pr_2^*\Omega)+\sum_{j>i} \pr_{j+1}^*\Omega\\
&=\sum_{j=1}^{l}\pr_j^*\Omega=\Omega^{(l)}. 
\end{align*}
\end{proof}

\begin{proof}[Proof of Theorem \ref{thm:local:global:forms}]
Let us start by remark that each $G^{(l)}$ is a manifold and the restriction of the map $\Phi:W(G)\to \AsCo(G)$ to each $G^{(l)}$ is a submersion. For $w\in G^{(l)}$ one finds:
\[ \Ker\d_w\Phi=\bigoplus_i \Ker\d_w \partial_i, \]
where the sum is over those $1\le i\le l-1$ for which $w\in\U^{(l)}_i$. By the previous lemma, one has that:
\[ i_v\Omega^{(l)}=0, \quad \text{if } v\in\Ker\d_w \partial_i. \]
Hence, for each $w\in G^{(l)}$ we obtain a form $\widetilde{\Omega}^w\in \wedge^kT_{\Phi(w)}\AsCo(G)$ such that:
\[ \Omega^{(l)}_w=(\d_w \Phi)^*\widetilde{\Omega}^w. \]
If $g=\Phi(w_1)=\Phi(w_2)$ then $w_1\sim w_2$, and applying successively \eqref{eq:face:multipl} we conclude that we must have $\widetilde{\Omega}^{w_1}=\widetilde{\Omega}^{w_2}$. It follows that one has well-defined k-form $\widetilde{\Omega}$ in $\AsCo(G)$, which satisfies:
    \[ \Phi^*\widetilde{\Omega}|_{G^{(l)}}=\Omega^{(l)}. \]
    This shows that $\widetilde{\Omega}$ is a smooth differential $k$-form.

Notice that for $l=1$ the map $\Phi:G\to \AsCo(G)$ is just the completion map, so $\widetilde{\Omega}$ also also extends $\Omega$, and we are only left to check that it is multiplicative. 

Multiplicativity follows by observing that for any $l_1,l_2\ge 1$ we have a commutative diagram:
\[
\xymatrix{
G^{(l_1)}\times G^{(l_2)}\ar[r]^{J} \ar[d]_{\Phi\times\Phi} & G^{(l_1+l_2)}\ar[d]^{\Phi}\\
\AsCo(G)^{(2)}\ar[r]_{\gm} & \AsCo(G)
}
\]
where $J$ is juxtaposition of words. From the definition of $\Omega^{(l)}$ one finds that:
\[ J^* \Omega^{(l_1+l_2)}=\pr_{G^{(l_1)}}^*\Omega^{(l_1)}+\pr_{G^{(l_2)}}^*\Omega^{(l_2)}, \]
and this implies the multiplicativity of $\widetilde{\Omega}$.

Finally, the uniqueness of the extension $\Omega$ follows from the fact that (i) if $G$ is $\t$-connected then so is $\AsCo(G)$ and (ii) two multiplicative forms on a $\t$-connected Lie groupoid which coincide in some neighborhood of the identity must actually coincide. 
\end{proof}

One can also use Theorem \ref{thm:local:global:forms} to recover the integrability of IM-forms to multiplicative forms (or other infinitesimal multiplicative structures) as stated in Theorem \ref{thm:linear:form:IM:forms}. Given a Lie groupoid $\G\tto M$ with 1-connected $\t$-fibers, one first integrates the IM-form to a multiplicative form in some neighborhood $M\subset U\subset \G$ as in Section \ref{sec:local:integration:forms}. Then applies Theorem \ref{thm:local:global:forms} using the following result:

\begin{proposition}
Let $\G\tto M$ be a Lie groupoid with 1-connected $\t$-fibers. Then any neighborhood  $M\subset U\subset \G$ is full.
\end{proposition}

\begin{proof}
Let $U$ be as in the statement and viewed it as local Lie groupoid $U\tto M$ obtained by shrinking $\G$. The inclusion $i:U\hookrightarrow \G$ induces a Lie groupoid morphism:
\[ \widehat{i}:\AsCo(U)\to \G, \]
which is a surjective local diffeomorphism. The Lie algebroid morphism induced by $\widehat{i}$ is just the identity. We claim that $\AsCo(U)$ has connected $\t$-fibers,  which together with the assumption that $\G$ has 1-connected $\t$-fibers, implies that $\widehat{i}$ must be an isomorphism, proving that $U$ is a full neighborhood.

To prove the claim, choose $M\subset V\subset U$ a neighborhood such that $\t^{-1}(x)\cap V$ is connected for all $x\in M$ and $V^{-1}U\subset U$. Since $\G\tto M$ is $\t$-connected, by Proposition \ref{prop:generate:inverses}, $V$ generates $U$. Given any element $g=[u_1,\dots,u_n]\in \AsCo(U)$ with $\t(g)=x$,  we can factor each $u_i$ as:
\[ u_i=v^i_{1}(v^i_{2}(\cdots(v^i_{k_i-1}v^i_{k_i})))\qquad (v^i_j\in V). \]
We conclude that we can write $g$ as:
\[ g=[v_1,\dots,v_N] \quad (v_i\in V). \]
Now choose paths:
\[ \gamma_i:[0,1]\to V,\quad \gamma_i(0)=1_{\s(v_{i-1})},\ \gamma_i(1)=v_i,\ \t(\gamma_i(t))=\s(v_{i-1}). \]
Then, we obtain paths $g_i:[0,1]\to \AsCo(U)$ in $\t^{-1}(x)$ by setting
\[ g_i(t):=[v_1,\dots,v_{i-1},\gamma_i(t)]. \]
Each such path joins $[v_1,\dots,v_{i-1},v_i]$ to $[v_1,\dots,v_{i-1},1_{\s(v_{i-1})}]=[v_1,\dots,v_{i-1}]$. The concatenation of these paths is a path in $\t^{-1}(x)\subset \AsCo(U)$ joining $g$ to $1_x$.
\end{proof}


\bibliographystyle{amsplain}
\bibliography{bibliography}
\end{document}